\documentclass[12pt,a4paper]{amsart}
\usepackage{amsmath, amsfonts, xifthen, latexsym, amssymb, amsthm, amscd}
\usepackage[utf8]{inputenc}
\usepackage{graphicx}
\usepackage{url}
\usepackage{float}
\usepackage[center]{caption}
\usepackage{hyperref}
\hypersetup{colorlinks=true,citecolor=blue,filecolor=blue,linkcolor=blue,urlcolor=blue}
\usepackage[margin=1.4in]{geometry}

\newtheorem{theorem}{Theorem}
\newtheorem{lemma}{Lemma} [section]
  
\newtheorem{proposition}[lemma]{Proposition}

\newtheorem{corollary} [lemma]{Corollary}

\theoremstyle{definition}
\newtheorem{definition}[lemma]{Definition}
\newtheorem{question}{Question}

\newtheorem{remark}{Remark}

\newtheorem{example}{Example}

\newcommand{\eps}{{\epsilon}}

\renewcommand{\phi}{{\varphi}}
\newcommand{\grd}{{Gr_d}}
\newcommand{\rgrd}{{RGr_d}}

\newcommand\dist{\operatorname{dist}}

\newcommand\Sch{\operatorname{Sch}}

\newcommand{\Stab}{\operatorname{Stab}}
\newcommand{\Supp}{\operatorname{Supp}}
\newcommand{\Haus}{\operatorname{Haus}}
\newcommand\diam{\operatorname{diam}}
\newcommand\Sep{\operatorname{Sep}}

\newcommand{\cP}{\mathcal{P}}\newcommand{\cM}{\mathcal{M}} \newcommand{\cF}{\mathcal{F}} \newcommand{\cl}{\mathcal{L}}
\newcommand{\cO}{\mathcal{O}}

\newcommand\R{\mathbb R}

\newcommand\Z{\mathbb Z}

\newcommand\N{\mathbb N}

\newcommand\cG{\mathcal{G}} \newcommand\cT{\mathcal{T}}
\newcommand{\cA}{\mathcal{A}}

\newcommand\file{\phi_{\lambda,\eps}}
\newcommand\pile{P_{\lambda,\eps}}
\newcommand{\actson}{\curvearrowright}

\newcommand{\cC}{\mathcal{C}} \newcommand{\cL}{\mathcal{L}}
\newcommand{\cZ}{\mathcal{Z}}
\newcommand{\cH}{\mathcal{H}}

\newcommand{\cal}[1]{{\mathcal #1}}

\newcommand{\vi}{\vskip 0.1in \noindent}
\usepackage{etoolbox}
\newtoggle{no_cases}\toggletrue{no_cases}
\newcommand{\case}[2][]{\iftoggle{no_cases}{\left\{\begin{array}{ll}#2 & #1}{\\#2 & #1}\togglefalse{no_cases}}
\newcommand{\esac}{\end{array}\right.\toggletrue{no_cases}}

\newcommand{\Prob}{\operatorname{Prob}}

\newcommand{\Sub}{\operatorname{Sub}}
\newcommand{\Spec}{\operatorname{Spec}}

\newcommand{\Cay}{\operatorname{Cay}}

\newcommand{\rsgd}{RSch_d}
\newcommand{\rcsd}{R\cC Sch_d}
\newcommand{\csgd}{R\cC Sch_d}

\begin{document}
\title[Strong almost finiteness]{Strong almost finiteness}

\thanks{The first author was partially supported by the KKP 139502 grant, the second author was partially supported by the Icelandic Research Fund grant number 239736-051 and the ERC Consolidator Grant 772466 ``NOISE''}
\author{G\'{a}bor Elek and \'{A}d\'{a}m Tim\'{a}r}
\begin{abstract} A countable, bounded degree graph is almost finite if it has a tiling with isomorphic copies of finitely many F\o lner sets, and we call it strongly almost finite, if the tiling can be randomized so that the probability that a vertex is on the boundary of a tile is uniformly small. We give various equivalents  for strong almost finiteness. In particular, we prove that Property A together with the F\o lner property is equivalent to strong almost finiteness. Using these characterizations, we show that graphs of subexponential growth and Schreier graphs of amenable groups are always strongly  almost finite, generalizing the celebrated result of Downarowicz, Huczek and Zhang about amenable Cayley graphs, based on graph theoretic rather than group theoretic principles. We give various equivalents to Property A for graphs, and show that if a sequence of graphs of Property A (in a uniform sense) converges to a graph $G$ in the neighborhood distance (a purely combinatorial analogue of the classical Benjamini-Schramm distance), then their Laplacian spectra converge to the Laplacian spectrum of $G$ in the Hausdorff distance. We apply the previous theory to construct a new and rich class of classifiable $C^{\star}$-algebras. Namely, we show that for any minimal strong almost finite graph $G$ there are naturally associated simple, nuclear, stably finite $C^{\star}$-algebras that are classifiable by their Elliott invariants. 
\end{abstract}
\maketitle
\newpage
\tableofcontents
\newpage
\section{Introduction}
\subsection{Motivations} Amenability was first introduced by John von Neumann in 1929 \cite{vonNeumann} in the setting of discrete groups. A group is called amenable if it admits an invariant mean. In graph theoretic terms this is equivalent to the existence of F\o lner sets (sets of arbitrarily small relative boundary) in its Cayley graph \cite{Folner}. Amenable groups play an important role in both the theory of von Neumann algebras and measurable equivalence relations. The concept of amenability was also developed for von Neumann algebras, where the analogue of the invariant mean is given by the conditional expectation onto the algebra. 
The group von Neumann algebra turned out to be amenable if and only if the group itself is amenable. 
Similarly, the measurable equivalence relation associated to a free probability measure-preserving (p.m.p.) action of a group is measurably amenable if and only if the group is amenable. 
A concept of finite approximability, later referred to as hyperfiniteness, was introduced for von Neumann algebras in 1943 by Murray and von Neumann \cite{Murray}. 
One of the major 
breakthroughs of the 1970's was Connes’ result \cite{Connes} showing that
for von Neumann algebras with separable preduals, amenability is equivalent to hyperfiniteness. 
Almost at the same time Connes, Feldman and Weiss \cite{CFW} proved a very similar result for the measurale equivalence relation associated to probability measure preserving (p.m.p) group actions: measurable amenability is equivalent to measurable hyperfiniteness, where measurable hyperfiniteness means that the equivalence relation associated to the action can be approximated in measure by finite equivalence relations, as introduced by Dye. 
Weiss conjectured that amenability and hyperfiniteness are also equivalent in the Borel setting. However, this conjecture remains open; in particular, it is still unknown whether free 
Borel actions of amenable groups are hyperfinite—that is, whether the associated Borel equivalence relation is hyperfinite.

\vi The topological setting (when the group is acting on a compact metric space instead of a probability space) is more subtle.
Any free, continuous action of an amenable group on a compact metric space is topologically amenable and admits an invariant probability measure. Moreover, only amenable groups can admit such actions. However, the class of countable groups that have a free, topologically amenable action on a compact metric space, i.e. the so-called Property A (or exact) groups, also contains nonamenable groups. Thus in the topological setting, amenability appears in two forms: classical amenability and Property A. 
The first example of a group that is not Property A was constructed only in 2003 by Gromov \cite{Gromov} (see also \cite{Osajda}). The notion was introduced by Yu \cite{Yu}, who proved that the Novikov conjecture holds for compact manifolds with fundamental group of Property A. It was soon proved that amenable groups, hyperbolic groups \cite{Roe}, linear groups \cite{Guentner} are of Property A. Ozawa \cite{Ozawa} proved that Property A is equivalent to the exactness of the reduced $C^\star$-algebras of the group. If the group is amenable then the reduced $C^\star$-algebra is even nuclear (that is, amenable \cite{Haagerup}).
\vi
What is the topological analogue of hyperfiniteness? Suppose a countable group acts continuously on the Cantor set, with the property that the fixed point set of each group element is clopen (as in the case of free actions). For such topological actions (or more precisely, for the associated étale Cantor groupoid), Matui introduced the notion of almost finiteness \cite{Matui}.  
If a group acts continuously on the Cantor set and the stabilizer map is also continuous, then the resulting groupoid is an ample étale groupoid. This perspective forms the basis for how we approach ample étale groupoids in our paper. As a matter of fact, all the groupoids in our paper are constructed in such a way. In particular, we can talk about their topological amenability through the respective definition for actions.
Almost finiteness for these groupoids means that the Cantor set has a tiling by clopen sets that are F\o lner in each orbit graph, making this property a good candidate for hyperfiniteness in the topological context. This perspective is further supported by two queries posed in (Remark 3.7, \cite{Suzuki}) of Suzuki, which suggest that such clopen F\o lner tilings could indicate a form of amenability: 
\begin{itemize}
\item
Is every minimal, almost finite, étale Cantor groupoid topologically amenable?
\item
Is every minimal, topologically amenable, étale Cantor groupoid that admits an invariant measure necessarily almost finite?
\end{itemize}
\noindent
For groupoids arising from actions of amenable groups, the answer to the first question is affirmative; however, the first author provided a counterexample in the more general setting \cite{Elekquali}.
It was later observed that a slight strengthening of almost finiteness does, in fact, imply amenability. Specifically, suppose a groupoid is not only almost finite, but also satisfies the following condition: for every $\varepsilon > 0$, there exists a collection of almost finite tilings equipped with a probability distribution such that, for each point $x$ in the Cantor set, the probability that $x$ lies on the boundary of a tile is less than $\varepsilon$. Under this strengthened version, which we call strong almost finiteness, the groupoid is topologically amenable. The second question by Suzuki is still open.
\vi
One of the main motivations of our paper is to provide further evidence that strong almost finiteness is, in fact, the appropriate continuous analogue of hyperfiniteness. Here a concept of almost finiteness for infinite graphs will become important.
\vi
Let us assume that an étale Cantor groupoid arises from an action of a finitely generated group. If the groupoid is almost finite, then its orbit graphs must also exhibit almost finiteness in the sense of tileability by F\o lner sets, as defined in the influential paper of Downarowicz, Huczek, and Zhang \cite{Down}, where they established the almost finiteness of Cayley graphs for countable amenable groups. Similarly, if the minimal groupoid is topologically amenable, then its orbit graphs possesses a graph theoretical version of Property A (which is equivalent to the group theoretical definition in case of Cayley graphs, and was defined by Higson and Roe \cite{Higson}). These are straightforward consequences of the definitions and the theorem by Connes, Feldman and Weiss \cite{CFW}.
Moreover, if the groupoid admits an invariant measure, the graphs should contain F\o lner sets in a ubiquitous fashion—an idea previously explored by Ma under the term ubiquitous amenability \cite{Ma}. Hence, in light of Suzuki's questions, it is reasonable to expect that Property A and the ubiquitous presence of F\o lner sets together are equivalent to strong almost finiteness in the context of bounded degree graphs.
We will prove that this equivalence indeed holds, establishing the equivalence of the suitably defined concepts of amenability and hyperfiniteness for bounded degree graphs. Furthermore, we characterize the relationship among other reasonable candidates for the definition of these concepts in this setup.
\vi
An additional motivation for our work was the result of Ma and Wu (\cite{Mawu}, Corollary 9.11) that the reduced $C^*$-algebra of an almost finite and amenable ample étale Cantor minimal groupoid is a simple, nuclear, unital, separable $\cal{Z}$-stable and quasidiagonal $C^\star$-algebra (that satisfies the Universal Coefficient Theorem by a theorem of Tu \cite{Tu1}). Hence, by the seminal result of Tikuisis, White and Winter \cite{quasi}, these $C^\star$-algebras can be classified by their Elliott invariants. Consequently, our results enable the construction of numerous new examples of classifiable, simple, nuclear $C^\star$-algebras arising built from strongly almost finite Schreier graphs.
\subsection{Around amenability. The key concepts of the paper}
\vi
Below,  $\grd$ will denote the set of all countable (not  necessarily connected) graphs of vertex degree bound $d$.
\begin{enumerate}   \setcounter{enumi}{-1}
\item  Let \( G \in \grd \) be a graph, and let \( F \) be a subset of its vertex set \( V(G) \). We denote by $\partial(F)$ the set of vertices in $F$ that are adjacent to a vertex that is not in $F$. If $\frac{|\partial(F)|}{|F|}<\eps$ then $F$ is called an \textbf{$\eps$-F\o lner set.} We call the graph $G\in\grd$ \textbf{amenable} if it contains a F\o lner sequence, that is, a sequence of finite subsets $\{F_n\}^\infty_{n=1}$ such
that $\frac{|\partial(F_n)|}{|F_n|}\to 0$. 
Amenable graphs were studied already in the eighties (see e.g. \cite{Dodziuk}, earlier work of Kesten \cite{Kesten} is frequently viewed as the first 
instance where purely graph theoretical properties were used in the context of amenability) 
and various characterizations of amenability for graphs, 
similar to the one of von Neumann, were given in the eighties and nineties (see e.g. \cite{Block}, \cite{Ceccherini}\cite{Dodziuk}, \cite{Elek1} or \cite{Juschenko} for a survey).
\item The graph $G$ 
is \textbf{F\o lner} if for any $\eps>0$ there exists an $r$ such that any $r$-ball $B^G_r(x)$ of radius $r$ centered around $x$ contains an $\eps$-F\o lner subset with respect to $G$. It is quite clear that the Cayley graphs of amenable groups are F\o lner graphs. Amenable, but non-F\o lner graphs can
easily be constructed by attaching longer and longer paths to the vertices of a tree, or just taking the disjoint union of a $3$-regular tree and an infinite path.
\item A graph $G\in\grd$ is \textbf{setwise F\o lner} if for any $\eps>0$ there is an $r>1$ such that inside the $r$-neighborhood $B_r(L)$ of any finite set $L\subset V(G)$ there exists an $\eps$-F\o lner set with respect to $G$ that contains $L$. As far as we know, this definition is new. 
\item A graph $G\in\grd$  is \textbf{uniformly locally amenable} if
for any $\eps>0$ there exists $k>0$ satisfying the
following condition:  For any finite subset $L\subset V(G)$ there exists a subset $M\subset L$ such that
$M$ is 
$\eps$-F\o lner with respect to L (so $M$ is not
necessarily $\eps$-F\o lner in the graph G) and $|M|\leq k$. The notion of uniform local amenability was introduced in \cite{Brodzki}. 
\item Building on Dye’s definition for measure-preserving actions (\cite{Dye}, see also \cite{Kechris}), the first author extended the concept of hyperfiniteness to classes of finite graphs with bounded degree \cite{Elekcost} in the following manner. A family of finite graphs \( \mathcal{G} \subset \grd \) is said to be \textbf{hyperfinite} if for every \( \varepsilon > 0 \), there exists an integer \( k > 0 \) such that for every \( G \in \mathcal{G} \), there exists a subset \( L \subset V(G) \) with \( |L| \leq \varepsilon |V(G)| \), such that removing \( L \) along with all incident edges results in a graph whose connected components each have at most \( k \) vertices. Note that the class of planar graphs, as well as any class of graphs with uniform subexponential growth, is hyperfinite.  The notion of local hyperfiniteness was introduced in \cite{Elekula}.  A graph $G\in\grd$ is \textbf{locally hyperfinite}, if the family of all its finite induced subgraphs  is hyperfinite.
\item  A graph $G\in\grd$ is \textbf{weighted hyperfinite} if for any $\eps>0$ there exists $k>0$ satisfying the following condition:
For any finitely supported non-negative function $w:V(G)\to \R$ there exists a subset $L\subset V(G)$ of
total weight $w(L)$ that is at most $\eps w(V(G))$, such that if we delete
$L$ with all the adjacent edges then the size of the
the remaining components are at most $k$. The notion of weighted hyperfiniteness  was introduced by the authors of this paper in \cite{Elektimar}. 
\item A subset $Y\in V(G)$ is a $k$-separator of the graph $G\in\grd$ if deleting $Y$ (with all the adjacent edges) the remaining components have size at most $k$.  A graph $G\in\grd$ is \textbf{strongly hyperfinite} if for any $\eps>0$ there exists $k>0$ and a probability measure $\mu$ on the compact set of $k$-separators (with the compact subset topology on $V(G)$) such that for all $x\in V(G)$, 
$$\mu(\{Y\,|\, x\in Y\})< \eps\,.$$
\noindent
The notion of strong hyperfiniteness appeared first in \cite{Zivny} in a somewhat restricted context.
\item An $(\eps,r)$-packing of a graph $G\in\grd$ is a family
of disjoint $\eps$-F\o lner sets of diameter at most r. A graph $G\in\grd$  is \textbf{strongly F\o lner hyperfinite} if for any $\eps>0$ there exists $r>0$ and a probability measure $\nu$ on the compact set of $(\eps,r)$-F\o lner packings $P(\eps,r)$ (we give the precise definition of the topology on $(\eps,r)$-F\o lner packings in Section \ref{tsct}) such that for all $x\in V(G)$, 
$$\nu (\{\cP\,|\, x\in \tilde{\cP}\})>1-\eps\,,$$
\noindent
where $\tilde{\cP}$ denotes the set of vertices contained in the elements of the packing $\cP$.
The notion of strong F\o lner hyperfiniteness is a crucial new notion of our paper.
\item A graph $G$ is \textbf{almost finite} if for any $\eps>0$ there exists an $r>1$ such 
that $V(G)$ can be tiled by $\eps$-F\o lner sets of diameter at
most r, in other words, there exists an $(\eps,r)$-packing $\cP$ such that $\tilde{\cP}=V(G)$. As we mentioned earlier, Downarowicz, Huczek and Zhang \cite{Down} defined almost finiteness (under the name of "tileability") for groups and showed that the Cayley graph of a finitely generated amenable group
is almost finite. Their work was motivated by the monotileability problem studied by Weiss \cite{Weiss}.  The notion of almost finiteness (in the case of free continuous actions of amenable groups) and its relation to $C^\star$-algebras was further developed in the important paper of Kerr \cite{Kerr}. Almost finite graphs, as in our setup, were introduced by Ara et al. \cite{Ara}.
\item A graph $G$ is \textbf{strongly almost finite} if for any $\eps>0$ there exists $r>1$
and a probability measure on the $(\eps,r)$-F\o lner tilings such that for any $x\in G$ the probability that $x$ is on the boundary of the tile containing it is less than $\eps$. This notion was introduced in a significantly weaker form by the first author in \cite{Elekquali}.
\item
\vi
 For graphs (and even for more general metric spaces) Property A was introduced in \cite{Higson}  (see also \cite{Brodzki}): a graph $G\in\grd$ is
of \textbf{Property A} if for any $\eps>0$ there exists $r>1$ and
a function $\Theta:V(G)\to \Prob(G)$ satisfying the following conditions.
\begin{itemize}
\item For every $x\in V(G)$ the support of $\Theta(x)$ is contained in the $r$-ball around $x$.
\item For every adjacent pair $x,y\in V(G)$ we have that
$$\|\Theta(x)-\Theta(y)\|_1<\eps\,.$$
\end{itemize}
\noindent
\item For a graph G the finitely supported non-negative (non-zero) function $f:V(G)\to \R$ is
an \textbf{$\eps$-F\o lner function} if
$$\sum_{x\in V(G)} \sum_{y, x\sim y} |f(x)-f(y)|<
\eps \sum_{x\in V(G)} f(x)\,.$$
\noindent
If $\sum_{x\in V(G)} f(x)=1$ we call such functions $\eps$-F\o lner probability measures.
\noindent
The role of $\eps$-F\o lner functions will be crucial in our paper. Note that in case of Cayley graphs of amenable groups the notions of  F\o lner functions and Reiter functions (see e.g. Theorem 2.16 \cite{Juschenko} for the definition of Reiter functions) are closely related.
A graph $G$ is of \textbf{F\o lner Property A} if it is of Property A and for all $x$ , $\Theta(x)$ can be chosen as a $\eps$-F\o lner function. 
\item The graph $G\in\grd$ is \textbf{fractionally almost finite}
if for any $\eps>0$ there exists $r\geq 1$ and a non-negative function 
$F:F_G(\eps,r)\to \R$, from the set of $\eps$-F\o lner sets
of diameter less than $r$ such that
that
\begin{enumerate}
\item For any $x\in V(G)$
$$\sum_{x\in H, H\in F_G(\eps, r)} F(H)
+c_x=1\,,$$
\noindent
where $0\leq c_x <\eps.$
\item For any $x\in V(G)$
$$\sum_{x\in \partial(H), H\in F_G(\eps, r)} F(H)<\eps.$$
\end{enumerate}
\end{enumerate}
This definition is motivated by Lov\'asz's notion of fractional partition \cite{Lovasz}.
\subsection{The results}
Some of the above properties have long been central in group theory. When one goes beyond Cayley graphs, a more complex scene emerges.
We fully investigate the relationships between properties $(1)-(12)$.  Figure 1 summarizes our results.
\vi
\begin{theorem} [The Long Cycle Theorem]\label{longcycle}
For graphs $G\in \grd$: Property A, Uniform Local Amenability, Local Hyperfiniteness, Weighted Hyperfiniteness and Strong Hyperfiniteness are equivalent. 
\end{theorem}
\noindent
Using some results from \cite{Brodzki} and \cite{Chen}, Sako \cite{Sako} has already established the equivalence between Property A and weighted hyperfiniteness. Building on the results of \cite{Sako}, \cite{Brodzki}, and \cite{Zivny}, the first author \cite{Elekula} has also shown that uniform local amenability is equivalent to Property A, confirming a conjecture proposed in \cite{Brodzki}. Nonetheless, the proof of Theorem \ref{longcycle} is presented in a self-contained manner.
\vi
\begin{theorem}\label{prop1}
For any $\eps>0$ there exists a $\delta>0$ such that if $G\in \grd$ and
$p:V(G)\to \R$ is a $\delta$-F\o lner probability measure, then
there exists an $\eps$-F\o lner subset $H\subset V(G)$ inside
the support of p such that $p(H)>1-\eps$.
\end{theorem}
\vi Note (see Remark \ref{reiter}) that in the case of Cayley graphs this theorem is a quantitative strenghtening of Theorem 2.16 in \cite{Juschenko}. By the triangle inequality, the finite sum of $\eps$-F\o lner functions is always an $\eps$-F\o lner function.
So, they behave much better with respect to summation than $\eps$-F\o lner sets do with respect to taking union. Using this advantage of the F\o lner functions, we prove that being a F\o lner graph implies the setwise F\o lner Property (Proposition \ref{uniloc}). 
\vi
\begin{theorem}[The Short Cycle Theorem] \label{shortcycle}
For graphs $G\in \grd$ the following properties are equivalent: Property A plus Setwise F\o lner, Strong F\o lner Hyperfiniteness, Fractionally Almost Finiteness and F\o lner Property A.
\end{theorem}
\vi
The properties in Theorem \ref{longcycle} are weaker than the ones in Theorem \ref{shortcycle}, since F\o lner Property A trivially implies Property A. The remaining strict inclusions are explained next, and are summarized on the diagram of Figure 1. First, there exist F\o lner graphs that are not almost finite (Proposition \ref{amealm}).
\begin{example}
The $3$-regular tree is the simplest and earliest example of graphs that have Property A, but are not amenable, let alone almost finite. To see that it has Property A, pick an end and for each vertex take the averaged indicator function of the path of length $n$ from the vertex towards the end. 
 \end{example}
\begin{example}
Almost finiteness does not imply Property A. Let $G\in \grd$ be a graph that is not of Property A, say it contains an embedded expander sequence or it is the Cayley graph of a non-exact group. Attach infinite paths to each vertex of $G$. Then, the resulting graph $H$ is clearly almost finite, but it is not of Property A (see \cite{Ara} or the unpublished result of the first author \cite{Elekquali}). Observe that $H$ is a F\o lner graph.
\end{example}
\vi
\begin{theorem} \label{elsotetel}
Strong F\o lner Hyperfiniteness and Strong Almost Finiteness are equivalent properties.
\end{theorem}
\noindent
As we mentioned earlier, Downarowicz et al. ~ \cite{Down} proved that the Cayley graph of an amenable group is almost finite. The ingenious proof uses in a crucial way the fact that such graphs are based on groups. Putting together Theorem \ref{shortcycle} and Theorem \ref{elsotetel},  we extend this result to much larger graph classes: for Schreier graphs of amenable groups (Proposition \ref{schrei}) and for graphs of subexponential growth (Proposition \ref{subexpclass}).
\begin{figure}[H] \label{figure}
\includegraphics[width=\linewidth]  
{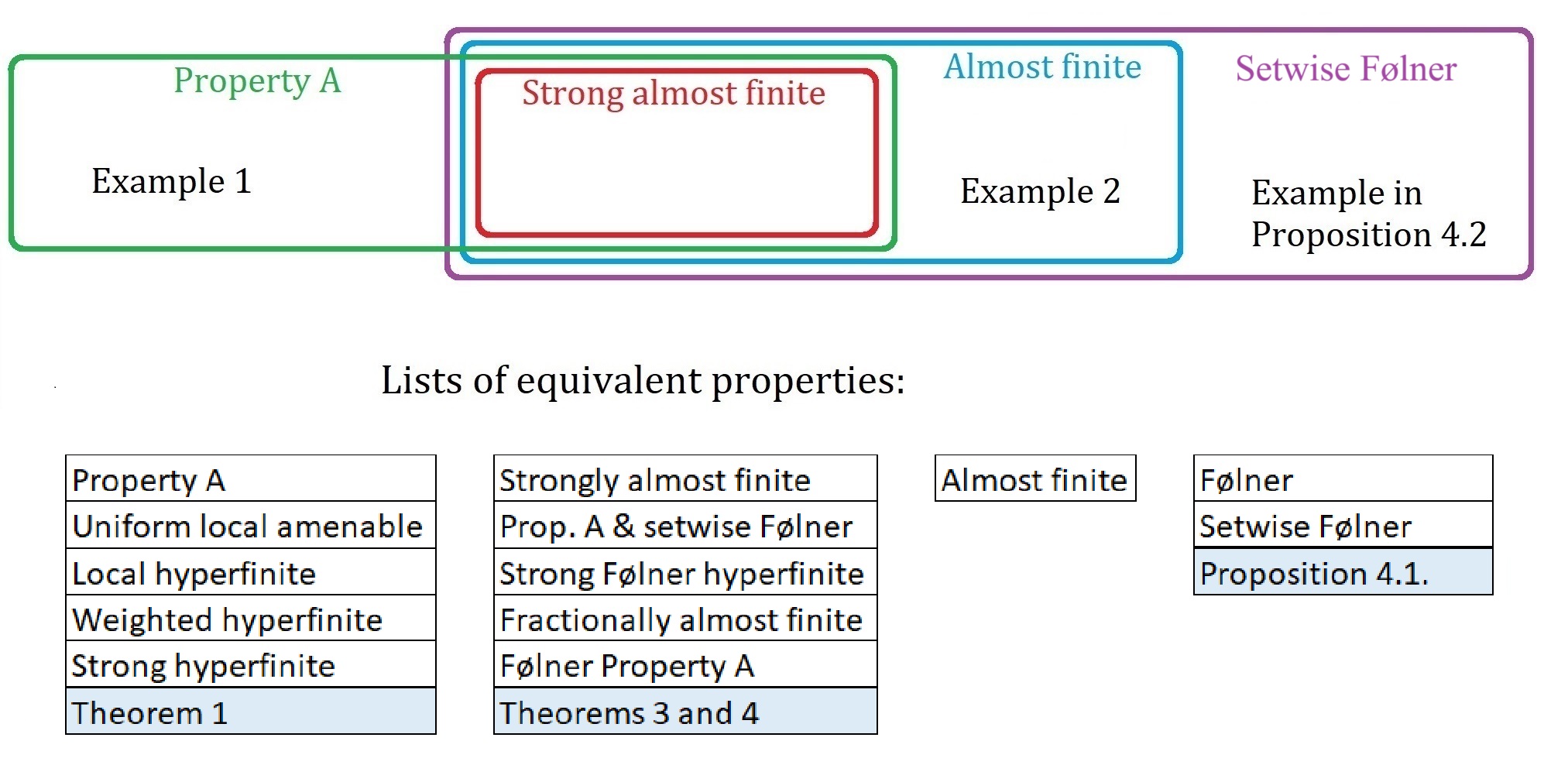}
\begin{center}
\caption{The relationships of the properties we study.}
\end{center}
\end{figure}
\noindent
We apply our results in spectral theory. 
Let $G\in\grd$ be a finite or infinite graph and
$\cal{L}_G:l^2(V(G))\to l^2(V(G))$ is its Laplacian. That is,
\begin{equation}
\cal{L}_G(f)(x)=\deg(x) f(x)-\sum_{x\sim y} f(y)\,. \end{equation}
\noindent
It is well-known that $\cl_G$ is a bounded, positive, self-adjoint operator and $\Spec(\cl_G)\subset [0,2d]$ (see \cite{Moharweiss}). It is well-known (see \cite{Kesten}, \cite{Dodziuk} or \cite{Moharweiss}) that a graph $G\in \grd$ is amenable  if and only if $0$ is in the spectrum of $G$. Note that the spectrum of Cayley graphs of amenable groups can be rather complicated \cite{Grigorchuk}.
It is a well-studied question that if a sequence of finite graphs $\{G_n\}^\infty_{n=1}$ converges to an infinite graph (or some other limit object) in some metric, what sort of convergence we can guarantee for the spectra $\{\Spec(\cL_{G_n})\}^\infty_{n=1}$. If the graphs  
$\{G_n\}^\infty_{n=1}$
are equipped with distinguished roots $\{x_n\in V(G_n)\}^\infty_{n=1}$ and the sequence of rooted graphs $\{(G_n,x_n)\}^\infty_{n=1}$ is convergent (see Proposition \ref{limitalmost}) then there exists
a rooted graph $(G,x)$ which is the limit of the sequence and the KNS-measures on $\{\Spec(\cL_{G_n})\}^\infty_{n=1}$ converge to the KNS-measure of $(G,x)$ in the weak topology (see \cite{Bartholdi}). Similar result is known (\cite{Abert}) if $\{\Spec(\cL_{G_n})\}^\infty_{n=1}$ is convergent in the sense of Benjamini and Schramm.
We will define neighborhood convergence (Section \ref{neigh}), a purely combinatorial version of the Benjamini-Schramm convergence and we prove the following theorem.
\vi
\begin{theorem} \label{spectralconv}
Say that a countable collection of graphs has Property A if their disjoint union is a graph with Property A.
Let $\{G_n\}^\infty_{n=1}\subset \grd$ be a countable set of
graphs of Property A such that $\lim_{n\to\infty}G_n \to G$ in the neighborhood distance. Then, $\Spec (\cL_{G_n})\to \Spec(\cL_G)$ in the
Hausdorff distance.
\end{theorem}
\noindent 
In the final section we establish a connection
between our strong almost finiteness property and the Elliott Classification Program on simple, nuclear $C^*$-algebras. We will show that if a graph $G\in\grd$ is minimal (see Definition \ref{minimal}) and $G$ is both of Property A and almost finite (that is $G$ is strongly almost finite) then some of the \'etale groupoids (see Section \ref{class} for the definition) which are
naturally associated to $G$ are minimal, topologically amenable and almost finite in the sense of Matui \cite{Matui}. Consequently, by the results in \cite{Mawu} we have the following theorem which we explain in Section \ref{class}.
\vi
\begin{theorem} \label{elliott}
For every minimal, strongly almost finite graph $M$ we can associate a stable action $\beta_E:\Gamma_{2d}\actson E$ so that all the orbit graphs are neighborhood equivalent to $M$ and
the simple, nuclear, tracial groupoid $C^*$-algebra $C^*_r(\cG_{\beta_E})$ is
classifiable by its Elliott invariants.
\end{theorem}
\vi These examples seem to be significantly different from the known ones. 

\noindent Finally, we give a purely dynamical characterization of strong almost finiteness (Proposition \ref{cstar}) in the case of minimal graphs.
\subsection{An overview of the paper}
Let us finish the introduction with a short overview of the coming sections. In Section \ref{sec2} we prove the five equivalents of Property A for bounded degree graphs, going along the ``long cycle''. Section \ref{sec3} establishes the quantitative connection between Folner functions and Folner sets, while Section \ref{sec4} proves that being a F\o lner graph and being a setwise Folner graph are the same. Property A together with the setwise Folner property can be characterized in four different ways, as shown in Section \ref{sec5}, going along the ``short cycle''. These are also equivalent to strong almost finiteness, as proved in Section \ref{sec6}. In Section \ref{sec7} various classes of graphs are shown to be strong almost finite. In Section \ref{sec8} we define neighborhood equivalence and a metric on the resulting equivalence classes of graphs, which will be the framework for Section \ref{sec9}, where the pointwise convergence of the spectrum is shown for convergent sequences of graphs in this topology. In Section \ref{class}, after the necessary preparations, we associate an {\it \'etale groupoid} to {\it minimal} graphs, and establish its topological amenability and almost finiteness under the assumption that the graph was strongly almost finite. This gives rise to a new and rich class of classifiable $C^*$-algebras.

\section{The Long Cycle Theorem} \label{sec2}
\noindent The goal of this section is to prove Theorem \ref{longcycle}.
The way we prove the theorem is showing that:
\noindent
Property A $\Rightarrow$  Uniform local amenability $\Rightarrow$ local hyperfiniteness
$\Rightarrow$ weighted hyperfiniteness $\Rightarrow$ strong hyperfiniteness $\Rightarrow$
Property A.
\begin{proposition} \label{long1}
Property A implies Uniform Local Amenability.
\end{proposition}
\proof The proof is a simplified version of Lemma 7.2 in \cite{Eleklocal}. Let $G\in\grd$ be a countably infinite
graph of Property A. Pick a $\delta>0$ in
such a way that we can find an $\eps$-F\o lner set in the support of any $d \delta$-F\o lner function  $\zeta:V(H)\to \R$, where $H$ is an arbitrary  finite induced subgraph of $G$. Such a choice is possible by Theorem \ref{prop1}. Since $G$ is of Property A,
there exists $r>1$ and
a function $\Theta:V(G)\to \Prob(G)$ satisfying the following conditions.
\begin{itemize}
\item For any $x\in V(G)$ the support of $\Theta(x)$ is contained in the $r$-ball around the vertex x.
\item For any adjacent pair $x,y\in V(G)$ we have that
$$\|\Theta(x)-\Theta(y)\|_1<\delta\,.$$
\end{itemize}
\noindent
Now, let $H$ be an arbitrary finite induced subgraph of $G$.
For $x\in V(G)$, pick $\tau(x)\in V(H)$ in such a way
that $d_G(x,\tau(x))=d_G(x,H)$.
For $x\in V(H)$, let $\Omega(x)(z)=\sum_{t\in \tau^{-1}(z)} \Theta(x)(t).$  Note that $\tau^{-1}(z)$ denotes the set of  vertices
mapped to $z$ by $\tau$.
Then by definition, $\Supp \Omega(x)\subset V(H)$ and
for all $z\in V(H)$, $\Omega(x)(z)\geq 0$. Also,
$$\sum_{z\in V(H)}\Omega (x)(z)= \sum_{t\in V(G)} \Theta(x) (t)=1\,,$$
hence $\Omega:V(H)\to \Prob(H)$. Also, if $x,y\in V(H)$ are adjacent vertices,
then
$$\|\Omega (x)-\Omega (y)\|_1 \leq \delta. $$
Indeed,  
$$\|\Omega (x)-\Omega (y)\|_1=\sum_{z\in V(H)} |\Omega (x)(z)-\Omega (y)(z)|= $$
$$=\sum_{z\in V(H)} |\sum_{t\in\tau^{-1}(z)} \Theta(x)(t)-
\sum_{t\in\tau^{-1}(z)} \Theta(y)(t)|\leq $$ $$
\leq \sum_{z\in V(H)} \sum_{t\in\tau^{-1}(z)} |\Theta(x)(t)-\Theta(y)(t)|= $$
$$= \sum_{t\in V(G)} |\Theta(x)(t)-\Theta(y)(t)|=\|\Theta(x)-\Theta(y)\|_1\leq \delta.$$
\noindent
Observe that
\begin{equation} \label{eA1} \Supp (\Omega(x))\subset B^G_{2r}(x),
\end{equation}
\noindent
where $B^G_{2r}$ denotes the ball of radius $2r$ centered around $x$ in the graph $G$. Indeed,
if $\Omega(x)(z)\neq 0 $, then there exists $t\in \tau^{-1}(z)$
such that $\Theta(x)(t)\neq 0$. Hence, $d_G(t,x)\leq r$
and also, $d_G(t,z)\leq r$, since $d_G(t,z)\leq d_G(t,x)$ by the 
definition of $\tau$.  
That is, $d_G(x,z)\leq 2r$, so for any $x\in V(H)$ we have that \eqref{eA1} holds. 

\noindent
The following lemma finishes the proof
of our proposition.
\begin{lemma}  There exists a subset
$L\subset V(H)$ such that $|\partial_H(L)|\leq \frac{d \delta}{2} |L|$
and $|L|\leq R_{2r} $, where
$R_{2r}$ is the maximal size of the $2r$-balls in $G$.
\end{lemma}
\proof
By the inequalities above,
$$\sum_{x\in V(H)} \sum_{x\sim y} \|\Omega(x)-\Omega(y)\|_1\leq
\sum_{x\in V(H)} d \delta =$$
$$= \sum_
{x\in V(H)} d \delta \|\Omega(x)\|_1\,,$$
\noindent
where here and going forward the summand $y$ is required to be in $V(H)$.
Hence,
$$\sum_{z\in V(H)} \sum_{x\in V(H)} \sum_{x\sim y}
|\Omega(x)(z)-\Omega(y)(z)|\leq d\delta \sum_{z\in V(H)}\sum_{x\in V(H)} 
\Omega(x)(z)\,.$$
\noindent
Hence, there exists $z_0\in V(H)$ such that
$$\sum_{x\in V(H)} \sum_{x\sim y}
|\Omega(x)(z_0)-\Omega(y)(z_0)|\leq d \delta \sum_{x\in V(H)} \Omega(x)(z_0)\,.$$
\noindent
Thus, if we define the function $\zeta:V(H)\to\R$ by
$\zeta(x)=\Omega(x)(z_0)$, we have that
\begin{equation}\label{kedd4uj}
\sum_{x\in V(H)} \sum_{x\sim y}
|\zeta(x)-\zeta(y)|\leq d \delta \sum_{x\in V(H)} 
\zeta(x)\,.
\end{equation}
\noindent
That is $\zeta$ is a $d \delta$-F\o lner function on $V(H)$. So, by our assumption on $\delta$, we can find a $\eps$-F\o lner set $L\subset V(H)$
inside the support of $\zeta$ (that is,
inside $B^G_{2r}(z_0)$). Hence, $|L|\leq R_{2r}$, thus our lemma follows. \qed
\vi
The proposition follows from the previous lemma right away. \qed
\begin{proposition}\label{long2}
Uniformly locally amenable graphs are
locally hyperfinite.
\end{proposition}
\proof First, let us remark that if $G$ is uniformly
locally amenable, then for any $\eps>0$ there exists
$k>0$ such that all finite, induced subgraphs $H\subset G$ contain a \emph{connected induced} subgraph $L$, $|V(L)|\leq k$
such that
\begin{equation} \label{e14A}
|\partial_H(V(L)|\leq \eps |V(L)|.
\end{equation}
\noindent
Indeed, if for a subset $E\subset H$,
$|\partial_H(E)|\leq \eps |E|$, $|E|\leq k$, then at least one of
the induced graphs on the components of $E$
satisfies \eqref{e14A}.

\noindent
So, let $\eps>0$ and let $k>0$ be as above.
Set $H_1:=H$ and let $L_1$ be a connected subgraph
of $H_1$ such that $|V(L_1)|\leq k$ and
$|\partial_{H_1}(V(L_1))|\leq \eps |V(L_1)|\,.$
Now let $H_2$ be the induced graph on $V(H_1)\backslash V(L_1)$. We pick a connected subgraph $L_2\subset H_2$ such that $|V(L_2)|\leq k$ and
$|\partial_{H_2}(V(L_2))|\leq \eps |V(L_2)|\,.$
Inductively, we construct finite induced subgraphs $H_1\supset H_2 \supset\dots$ and connected subgraphs
$L_i\subset H_i$ such that
 $|V(L_i)|\leq k$ and
$|\partial_{H_i}(V(L_i))|\leq \eps |V(L_i)|$
(of course, for large enough $q$, $H_q$ and $L_q$ are empty graphs).

\noindent
Now, let $S:=\cup^\infty_{i=1} \partial_{H_i}(V(L_i))$. Then, if remove $S$ from $H$ together with all the incident edges, the remaining components have size at most $k$.
 \qed
\begin{proposition} \label{long3}
Locally hyperfinite graphs are weighted hyperfinite.
\end{proposition}
\proof Our proof is based on the one of Lemma 8.1 \cite{Zivny}. We call a finite graph $H$ $(\delta,k)$-hyperfinite if one can delete not more than $\delta |H|$ vertices of $H$ together with all the
incident edges such that the sizes of the remaining components
are not greater than $k$. Also, we call a finite graph $J$ $(\eps,k)$-weighted hyperfinite, if for all positive weight function $w:V(J)\to \R$ one can delete a set of vertices $S\subset V(J)$ with
total weight at most $\eps w(V(J))$  such that that the sizes
of the remaining components is at most $k$.
It is enough to prove that if $G\in\grd$, then for
any $\eps>0$ there exists $\delta>0$ such that if all the finite induced subgraphs of $G$ 
are $(\delta,k)$-hyperfinite than they are $(\eps,k)$-weighted hyperfinite as well.

\noindent
Fix $\eps>0$ and let $L$ be the smallest integer that is larger than $\frac{3}{\eps}$. Now, assume that
the finite induced subgraphs of the countably infinite graph $G\in \grd$
are $(\delta,k)$-hyperfinite, where
\begin{equation}
\delta=\left(\frac{\eps}{3d}\right)^{L-1}\frac{\eps}{3}\,.
\end{equation}
\noindent
Let $H$ be a finite induced subgraph of $G$.
We partition the vertices of $H$ into
subsets $B_i, i\in \Z$ such that
$$B_i=\{ x\mid \left(\frac{\eps}{3d}\right)^{i+1}<w(x)\leq \left(\frac{\eps}{3d}\right)^i\}\,.$$
For
$0\leq q \leq L-1$ consider the subset
$$D_q=\cup_{i\in \Z} B_{iL+q}\,.$$
Since $L>\frac{3}{\eps}$, there exists
some $0\leq m \leq L-1$ such that
\begin{equation} \label{elsobecslesA}
w(D_m)<\frac{\eps}{3} w(V(H)).
\end{equation}
Now, for $i\in \Z$ set
$$C_i=\cup_{q=1}^{L-1} B_{iL+m+q}\,.$$
\noindent
By our assumption, if $x\in C_i$ and $y\in C_j$, $i<j$, then 
\begin{equation}
\label{A12}
\left(\frac{\eps}{3d}\right) w(x)> w(y)\,.
\end{equation}
\noindent
For $i<j$ let $F_{i,j}$ be the
set of vertices $y$ in $C_j$ such that
there exists $x\in C_i$, $x\sim y$.
By the vertex degree assumption,
$|\cup_{i<j} F_{i,j}|\leq d |C_i|.$ Also, if $y\in \cup_{i<j} F_{i,j}$
and $x\in C_i$, then
$\frac{\eps}{3d} w(x)\geq w(y)\,.$
Therefore, $w(\cup_{i<j} F_{i,j})\leq
\frac{\eps}{3} w(C_i)\,.$ Consequently,
\begin{equation} \label{masodikbecslesA}
w(F)\leq \frac{\eps}{3}
w(V(H))\,,
\end{equation}
\noindent
where $F=\cup_{i\in \Z} \left(\cup _{i<j} F_{i,j}\right)$.
Let us delete the subsets
$F$ and $D_m$ from $V(H)$
together with all the incident edges.
Then, each of the remaining components are
inside of the subsets $C_i$. Let $T$
be such a component. By our assumption,
$T$ is $(\delta,k)$-hyperfinite, thus
one can delete a subset $S\subset V(T)$ so 
that $|S|\leq \delta |V(T)|$ in such a way
that all the remaining components are of
size at most $k$.
By the definition of $C_i$, if $y\in V(T)$,
$$\min_{x\in V(T)} w(x)\geq \left(\frac{\eps}{3d}\right)^{L-1} w(y)\,.$$
\noindent Hence, 
$$w(S)\leq |S| (\min_{x\in V(T)} w(x))\left(\frac{3d}{\eps}\right)^{L-1}\leq
\delta w(V(T)) \left(\frac{3d}{\eps}\right)^{L-1}=\frac{\eps}{3} w(V(T))\,.$$
By \eqref{elsobecslesA} and \eqref{masodikbecslesA},
$$w(F\cup D_m)<\frac{2\eps}{3} w(V(H))\,.$$
\noindent
So, by deleting a set of vertices of weight less than $\eps |V(H)|$, we obtained a graph that has components of size at most k. \qed
\begin{proposition} \label{long4}
If $G\in\grd$ is weighted hyperfinite then $G$ is strongly hyperfinite.
\end{proposition}
\proof First we prove the proposition for finite graphs. This part is based on the proof of Lemma 4.1 in \cite{Elekula}. A finite graph $G$ is $(\eps,k)$-strongly hyperfinite, if there exists a probability measure $\mu$ on $\Sep(G,k)$ such that for each $x\in V(G)$
$$\mu(\{Y\,\mid x\in Y\})\leq\eps\,.$$
\begin{lemma}\label{weight}
If a finite graph G is $(\eps,k)$-weighted hyperfinite, then it is $(\eps,k)$-strongly hyperfinite as well. 
\end{lemma}
\proof
Assume that $G$ is $(\eps,k)$-weighted hyperfinite and $V(G)=n$. Let $m$ be the number of $k$-separators. For a $k$-separator $Y$ let $\{\underline{c}_Y:V(G)\to \{0,1\}\}\in \R^n$ be its characteristic vector.
We define the hull $\cH$ of the $k$-separators as the convex set of vectors $\underline{z}\in \R^n$ which can be written in the form
$$\underline{z}=\sum_{i=1}^m x_i \underline{c}_{Y_i}+\underline{y}\,,$$
\noindent
where $\{x_i\}_{i=1}^m$ are non-negative real numbers, $\sum^m_{i=1} x_i=1$ and $\underline{y}$ is a non-negative vector.
Now, let $\underline{v}=\{\eps,\eps,\dots,\eps\}$.
We have two cases.

\noindent
{\bf Case 1.} $\underline{v}\in \cH$. Then, there exist non-negative real numbers $\{x_i\}^m_{i=1}$, $\sum^m_{i=1} x_i=1$   such that all the absolute values of the coordinates of the vector $\sum^m_{i=1} x_i \underline{c}_{Y_i}$ are less than or equal to $\eps.$
That is, if the probability measure $\mu$ on the $k$-separators is given by $\mu(Y_i)=x_i$, the $(\eps,k)$-strong hyperfiniteness follows.

\noindent
{\bf Case 2.} $\underline{v}\notin \cH$. Since $\cH$ is a closed convex set, there must exist a hyperplane $H\subset \R^n$ containing
$\underline{v}$ such that $\cH$ is entirely on one side of the hyperplane $H$.
That is, there exists a vector $\underline{w}\in \R^n$ such that for any $\underline{c}\in \cH$ we have
$$\langle \underline{w},\underline{v} \rangle < \langle \underline{w},\underline{c} \rangle\,.$$
\noindent
Clearly, for any $1\leq i \leq n$, $w_i\geq 0$, since the $i$-th coordinate of $\underline{c}$, can be
increased arbitrarily, while keeping the other coordinates fixed. We can also assume that $\sum_{i=1}^n w_i=1.$
So, 
$$\eps< \langle \underline{w}, \underline{c}_Y \rangle $$
\noindent
holds for any $k$-separator $Y$, that is, $G$ is not $(\eps, k)$-weighted hyperfinite, leading to a contradiction. \qed
\vi
\begin{lemma}
Let $G$ be a countably infinite graph of vertex degree $d$. Suppose that for any $\eps>0$ there exists $k>0$ such that all the finite subgraphs of $G$ are 
$(\eps,k)$-strongly hyperfinite. Then, $G$ is strongly hyperfinite. 
\end{lemma}
\proof Let $\{H_m\}^\infty_{m=1}$ be finite induced subgraphs in $G$ such that $V(H_1)\subset V(H_2)\subset\dots$ and $\cup_{m=1}^\infty V(H_n)=V(G)\,.$
Let $\nu_n$ be a probability measure
on $\Sep(H_n,k)$ such that for all $x\in V(H_n)$ we have
$$\nu_n(\{Y\in \Sep(H_n,k)\mid x\in Y\})\leq \eps\,.$$
For all $n\geq 1$ we have an injective
map 
$$\phi_n:\Sep(H_n,k)\to \Sep(G,k)$$
\noindent
mapping the $k$-separator $Y_n$ to 
$Y=Y_n\cup(V(G)\backslash V(H_n))\,.$
Now, let $\mu_n=(\phi_n)_\star(\nu_n)\,.$
Let $\mu_{n_k}\to\mu$ be a weakly convergent
subsequence.
For all $x\in V(G)$, let $U_x\subset \Sep(G,k)$
be the set of $k$-separators containing $x$. Clearly, $U_x$ is a closed-open subset of $\Sep(G,k)$. By our assumptions,
for large enough $k$, $\mu_{n_k}(U_x)\leq \eps$. Therefore, $\mu(U_x)\leq \eps$. Hence, our lemma follows. \qed
\vi
This finishes the proof of our proposition. \qed
\begin{proposition} \label{long5}
If $G\in \grd$ is strongly hyperfinite, then G is of \\Property A.
\end{proposition}
\proof Let $\eps>0$. Since $G$ is strongly hyperfinite,
there exists $k>0$ and a probability measure
$\mu$ on $\Sep(G,k)$ such that for all $x\in V(G)$
\begin{equation} \label{e13A}
\mu({Y\mid x\in Y})\leq \frac{\eps}{4}\,.
\end{equation}
\noindent
We define
a non-negative function $F:I_G(k)\to \R$, where
$I_G(k)$ is the set of induced, connected
subgraphs of $G$ having at most k vertices, in the following way.
Let $F(H)$ be defined as the $\mu$-measure of the set
of $k$-separators $Y$ such that $H$ is a component
in $G\backslash Y$.
 
\noindent
Now, let $p_H$ be the uniform probability measure
on $H$. Then, let 
$$\Theta_x=\sum_{H\in I_G(k), x\in H} F(H) p_H+ c_x\delta_x\,, $$
where $$c_x:=1-\sum_{H\in I_G(k), x\in H} F(H)\leq  \frac{\eps}{4}\,.$$
\noindent
Then, for all $x\in V(G)$, $\|\Theta_x\|_1=1$ and $\Supp(\Theta_x)\subset B^G_k(x)\,.$

\noindent
Now, let $x\sim y$ be adjacent vertices. Then,
$$\|\Theta_x-\Theta_y\|_1\leq
c_x+c_y +\sum_{H, x\in H, y\notin H} F(H) +
\sum_{H, x\notin H, y\in H} F(H)\leq \eps\,.$$
\noindent
Therefore, $G$ is of Property A. \qed
\vi
Now, by Propositions \ref{long1}, \ref{long2}, \ref{long3}, \ref{long4} and \ref{long5}, the Long Cycle Theorem follows. \qed

\section{F\o lner functions} \label{sec3}
\noindent The goal of this section is to prove Theorem \ref{prop1}.
\proof We will prove the existence of $\delta$ for a fixed graph $G\in\grd$. However, this is enough to prove our theorem. Indeed, assume that for some $\eps>0$ there is a sequence of graphs $\{G_n\}^\infty_{n=1}$ such that the
largest $\delta$'s satisfying the condition of our theorem tend to zero. Then, for the disjoint union $\cup^\infty_{n=1} G_n$ one could not pick a $\delta$ that satisfies the condition of our theorem. So, we fix $G\in \grd$ and $\eps>0$.
As follows, let $q$ be the smallest integer that is greater than $\frac{16}{\eps^2}\,.$
Let $\rho>0$ be
such that $(1+\rho)^{2q}=1.1$.  Let us begin the proof with a technical lemma.
\begin{lemma}\label{lemma1}
Let $p$ be a probability measure on the finite set $\{x_i\}_{i=1}^n\subset V(G)$ such that
for any $1\leq i \leq n$ we have $0<p(x_i)<\frac{1}{2}\,.$ Then there exist
positive real numbers $a_1<a_2<\dots< a_m$ satisfying the following conditions:
\begin{enumerate}
\item $a_1<\frac{\eps^2}{16n}$.
\item For any $1\leq i \leq m-1$, $a_{i+1}=2a_i\,.$
\item $\sum_{i=1}^{m-1} p(M_i)>1-\frac{\eps^2}{8}\,,$

\noindent where $M_i=\{x_j\,|\, (1+\rho)a_i<p(x_j)<(1+\rho)^{-1}a_{i+1}\}\,.$
\end{enumerate}
\end{lemma}
\proof Let $s<\frac{\eps^2}{32n}$ be a positive real number
such that for all non-negative integers $k,l,j$,
$2^ks(1+\rho)^j\neq p(x_l)$. Let $m$
be the largest integer such that $s 2^m<\frac{1}{2}\,.$
We define the disjoint subsets $T_1,T_2,\dots, T_q$
by 
$$T_j:= \bigcup_{i=1}^{m-1} \left(
2^{i-1}s(1+\rho)^{2j-2}, 2^{i-1}s(1+\rho)^{2j-1} \right)\cup
\left(  2^i s(1+\rho)^{2j-3}, 2^i s(1+\rho)^{2j-2}   \right)$$
\noindent
So, there exists $1\leq l \leq q$ such that
$$p(T_l)\leq \frac{1}{q}<\frac{\eps^2}{16}\,.$$
\noindent
Set $a_1:=s(1+\rho)^{l-1}<\frac{\eps^2}{16n}$, and define 
$a_i:=2^{i-1} s (1+\rho)^{l-1}$ as in $(2)$ of the lemma. Observe that
$\sum_{i\mid p(x_i)<a_1} p(x_i)< \frac{\eps^2}{16}$, hence
we have that $$\sum_{i=1}^{m-1} p(M_i)>1-p(T_l)-\sum_{i\mid p(x_i)<a_1} p(x_i)
>1-\frac{\eps^2}{8}\,.\quad \qed$$
\vi
We define the constant $\delta$ in the following lemma.
\begin{lemma}\label{lemma2}
Let $\delta=\frac{\eps^2(\sqrt{1+\rho}-1)}{16}$ and let  $p:V(G)\to \R$ be a $\delta$-F\o lner function, such that
$\sum_{x\in V(G)}p(x)=1$. Define $B(G)\subset V(G)$ as the set of vertices $x$ in the support of $p$ such that there exists $y\sim x$, so that $y$ is not in the support of p or there exists $y\sim x$, such that either $\frac{p(x)}{p(y)}>\sqrt{1+\rho}$ or
$\frac{p(y)}{p(x)}>\sqrt{1+\rho}$. Then,
\begin{equation} \label{e1}
\sum_{x\in B(G)} p(x)<\frac{\eps^2}{8}\,.
\end{equation}
\end{lemma}
\proof If $\frac{p(x)}{p(y)}>\sqrt{1+\rho}$,  then
$$|p(x)-p(y)|>(1-\frac{1}{\sqrt{1+\rho}}) p(x)>
\frac{\sqrt{1+\rho}-1}{2} p(x)\,.$$
If $\frac{p(y)}{p(x)}>\sqrt{1+\rho}$, then
$$|p(x)-p(y)|>(\sqrt{1+\rho}-1)p(x)\,.$$
\noindent
If $y$ is not in the support of $p$, then we also have that
$$|p(x)-p(y)|=p(x)> (\sqrt{1+\rho}-1)p(x)\,.$$
Hence, we have that
$$\sum_{x\in B(G)} \frac{\sqrt{1+\rho}-1}{2} p(x)
< \sum_{x\in B(G)}\sum_{x\sim y}|p(x)-p(y)|<\delta.$$
\noindent
Therefore, our lemma follows. \qed
\vi
Now let us consider our $\delta$-F\o lner probability measure $p$ and let $n$ be the size of support of $p$. We may assume that all the values of $p$ are smaller than $\frac{1}{2}$.
Pick the numbers $\{a_i\}^m_{i=1}$ to satisfy the conditions of Lemma \ref{lemma1}. For
$1\leq i \leq m-1$, let 
$b_i=\sqrt{1+\rho} a_i$, $c_i=(\sqrt{1+\rho})^{-1} a_{i+1}\,.$ Finally, let
$$S_i:=\{x_j\,\mid\, b_i\leq p(x_j) \leq c_i\}\,.$$\begin{lemma}
Let $Q=\cup_{i\mid\,S_i\,\mbox{is not $\eps$-F\o lner}} S_i$. Then, $p(Q)<\frac{\eps}{2}\,.$
\end{lemma}
\proof Observe that 
\begin{equation} \label{e91} p\left(\cup_{i=1}^{m-1} \partial (S_i)\right)<\frac{\eps^2}{4}.
\end{equation}
\noindent
Indeed,  $$\cup_{i=1}^{m-1} \partial (S_i)
\subset L\cup B(G)\,,$$
\noindent
where $L$ is the complement of $\cup_{i=1}^{m-1} M_i$ from Lemma \ref{lemma1}  and $B(G)$ is the set defined in Lemma \ref{lemma2}.
So, \eqref{e91} follows from Lemma \ref{lemma1} and Lemma \ref{lemma2}.

\noindent
Now, if $S_i$ is not $\eps$-F\o lner, then we have that
\begin{itemize}
\item
$|\partial(S_i)|\geq\eps |S_i|$.
\item $a_i |\partial(S_i)|<p(\partial(S_i))< 2a_i |\partial(S_i)|$.
\item $a_i |S_i|<p(S_i)< 2a_i |S_i|$.\
\end{itemize}
\noindent
Thus, $p(\partial(S_i))> a_i \eps |S_i|>\frac{\eps}{2} p(S_i).$
Hence,
$$\sum_{i\mid\,  S_i\,\mbox{is not $\eps$-F\o lner}} p(S_i) \leq \frac{2}{\eps}\frac{\eps^2}{4}<\frac{\eps}{2}. \quad \qed $$
\noindent
That is, $$\sum_{i\mid\,  S_i\,\mbox{is $\eps$-F\o lner}} p(S_i) >1-(\sum_{i\mid\,  S_i\,\mbox{is not $\eps$-F\o lner}} p(S_i))-(1-\sum_{i=1}^{m-1} p(M_i))>1-\eps.$$
\noindent
Therefore, if we set
$H=\cup_{i\mid\,  S_i\,\mbox{is $\eps$-F\o lner}} S_i$, our theorem follows. \qed
\vi
\begin{remark} \label{reiter}
In the case of Cayley graphs, Theorem 2.16 of \cite{Juschenko} entails the existence of an $\eps$-F\o lner set $E$ in the support of a $\delta$-F\o lner probability measure $p$, without any control on the measure of $E$.
\end{remark}

\section{F\o lner Graphs are Setwise F\o lner} \label{sec4}
\begin{proposition} \label{uniloc}
F\o lner Graphs are Setwise F\o lner.
\end{proposition}
\proof
Assume that there exists a F\o lner graph $G$ that is not setwise F\o lner.
That is, there exists $\eps>0$ such that for any $k\geq 1$ there is a finite subset $L\subset V(G)$
so that there is no $\eps$-F\o lner set in $B_k(L)$ that contains L.

\noindent
Given $\eps>0$, let $\delta=\delta(\frac{\eps}{2})$ be as in Theorem \ref{prop1}. So, for each $\delta$-F\o lner function $f$ there
exists an $\frac{\eps}{2}$-F\o lner subset H inside the support of $f$ so that
$$\sum_{x\in H} f(x)>(1-\frac{\eps}{2}) \sum_{x\in V(G)} f(x)\,.$$
\noindent
Let $r$ be a natural number such that for any $x\in V(G)$ the ball $B^G_r(x)$ contains a
$\frac{\delta}{d}$-F\o lner set.
Now, for all $k\geq 1$ we choose a natural number $T_k$ satisfying the inequality 
$$T_k \ln(1+\frac{\eps}{d})>2 \ln(k)\,,$$
that is,
\begin{equation} \label {e22}
(1+\frac{\eps}{d})^{T_k}> k^2.
\end{equation}
\noindent
By our assumptions, for every large enough $k\geq 1$ there exists a finite subset $L_k\subset V(G)$ such that there is no $\eps$-F\o lner set $H_k$ so that
$$L_k\subset H_k\subset B_{T_k+r}(L_k)\,.$$
\begin{lemma}
If $k$ is large enough then we have
\begin{equation} \label{e23}
|B_{T_k} (L_k)|> k R_r |L_k|\,,
\end{equation}
\noindent
where $R_r$ is the size of the largest ball of radius r in the graph G.
\end{lemma}
\proof
For $k\geq 1$ we consider the subsets
$$L_k\subset B_1(L_k) \subset B_2(L_k)\subset\dots\subset B_{T_k}(L_k)\,.$$
\noindent
By our assumptions, for $0\leq i \leq T_k$ $B_i(T_k)$ is not an $\eps$-F\o lner set, hence
$|B_{i+1}(L_k)|>(1+\frac{\eps}{d})|B_i(L_k)|$, that is
$$|B_{T_k} (L_k)|> (1+\frac{\eps}{d})^{T_k} |L_k|\,.$$
\noindent
Hence by \eqref{e22}, for large enough $k$ we have that
$$|B_{T_k} (L_k)|> k R_r |L_k|\,,$$
\noindent
so our lemma follows. \qed
\vi
Now, for each $y\in B_{T_k}(L_k)$ consider the uniform probability measure $p_y$ of a $\frac{\delta}{d}$-F\o lner set inside $B^G_r(y)$.
Clearly, $p_y$ is a $\delta$-F\o lner function. Therefore, $g=\sum_{y\in B_{T_k}(L_k)} p_y$ is
a $\delta$-F\o lner function as well, supported in the set $B_{T_k+r}(L_k)$.
So by our assumption, there exists an $\frac{\eps}{2}$-F\o lner subset $H_k\subset B_{T_k+r}(L_k)$
such that
$$\sum_{z\in H_k} g(z)> (1-\frac{\eps}{2}) \sum_{z\in B_{T_k+r}(L_k)} g(z)\,.$$
By definition, for any $z\in V(G)$, $g(z)\leq R_r$. That is,
$$|H_k| R_r > (1-\frac{\eps}{2}) \sum_{z\in B_{T_k+r}(L_k)} g(z)> (1-\frac{\eps}{2})  k R_r |L_k|\,.$$
\noindent
So, the inequality $|H_k|>\frac{k}{2} |L_k|$ holds provided that $k$ is large enough.
Therefore, for large $k$ values the subset $H_k\cup L_k$ is an $\eps$-F\o lner set inside
the subset $B_{T_k+r}(L_k)$ containing $L_k$, leading to a contradiction. \qed
\vi
\begin{proposition}\label{amealm} For sufficiently large $d$, 
there exist F\o lner graphs in $\grd$ that are not almost finite.
\end{proposition}
\proof
Let us pick a $c$-expander sequence of finite connected graphs $\{G_n\}^\infty_{n=1}$. That is, for some $0<c<1$ the following condition holds. If $n\geq 1$, $L_n\subset V(G_n)$ and $|L_n|\leq \frac{1}{2}|V(G_n)|$, then $|\partial(L_n)|\geq c|L_n|$. We also assume that
the diameter of $G_n$ is at least $n$. Now fix a sequence of integers
$\{a_n\}^\infty_{n=1}$ such that
\begin{equation}\label{expander}
n 2^{n+1} < a_n.
\end{equation}
\noindent
Now we pick pairs of vertices $x_k,y_k, d_{G_k}(x_k, y_k)=n$ and for each $k$ connect $y_k$ and $x_{k+1}$ by a new edge. Let $G$ be the resulting connected graph.
Now for any $n\geq 1$ pick a maximal subset $X_n$ of $V(G)$ satisfying the following conditions.
\begin{itemize}
\item If $x\neq y$ are elements of $X_n$ then $d_G(x,y)>2a_n$.
\item If $a_n>\frac{\diam(G_k)}{3}$, then $G_k\cap X_n$ is empty.
\end{itemize}
\noindent
Now for each $n\geq 1$ attach a path of length $n$ to all elements of $X_n$. 
If a vertex belongs to more than one set $X_n$, we attach only one path to it - the longest among them. 
Note that, by our conditions, a vertex can belong to only finitely many of the sets $X_n$.
The resulting graph
$G^0$ is F\o lner, since every vertex 
is at bounded distance from one of the attached paths of length $n$. 
For each $k\geq 1$, let $Y_k$ be the set
of vertices in $V(G^0)\backslash V(G)$ that are in a path that is attached to a vertex of $G_k$.
Also, let $X_n^k:=V(G_k)\cap X_n$. 
\begin{lemma}
For each $n\geq 1$ and $k\geq 1$, we have that
\begin{equation}\label{tech1}
a_n |X_n^k|\leq |V(G_k)|\,.
\end{equation}
\end{lemma}
\proof
By our condition, the  balls of radius $a_n$ centered around the elements of $X_n$ are disjoint. If $X^k_n\not=\emptyset$, then $\diam(G_k)\geq 3a_n$ by definition of $X^k_n$.
Also, if $x\in X_n^k$ then $|B^G_{a_n}(x)\cap V(G_k)|\geq a_n$. Indeed, let $z$ be one of the elements in $G_k$ that is farthest from
$x$. Then, the shortest path in $G_k$ from $x$ to $z$ contains at least $a_n$ elements contained in 
$B^G_{a_n}(x)$. Therefore, $a_n |X_n^k|\leq |V(G_k)|$. \qed
\vi
Hence by $\eqref{expander}$ and $\eqref{tech1}$, we have that $n |X_n^k|<\frac{1}{2^{n+1}} |V(G_k)|$. That is, we have that
\begin{equation}
|Y_k|<\frac{1}{4} |V(G_k)|.
\end{equation}
\noindent
Assume that $G^0$ is almost finite. Then, we have a tiling of $V(G^0)$ by $\frac{c}{10}$-F\o lner sets $\{T_i\}^\infty_{n=1}$ such that $\diam (T_i)<r$ for some integer $r$.
Clearly, if $k$ is large enough there exists a tile $T_i$ such that $T_i\cap V(G)=S$ is fully contained in $V(G_k)$ and $|T_i\cap Y_k|\leq \frac{|S|}{2}\,.$  Then $|\partial S|\geq c |T_i|\geq \frac{c}{2}|S|,$ in contradiction with the fact that $T_i$ is a $\frac{c}{10}$-F\o lner set. \qed

\section{The Short Cycle Theorem} \label{sec5} \label{tsct}
\noindent The goal of this section is to prove Theorem \ref{shortcycle}.
The way we prove the theorem is showing that:
\noindent
Property A + Setwise F\o lner $\Rightarrow$ Strong F\o lner hyperfiniteness  $\Rightarrow$ Fractional almost finiteness
$\Rightarrow$ F\o lner Property A $\Rightarrow$ 
Property A + Setwise F\o lner. 

\noindent First, let us define a compact metric space structure on the space $P_G(\eps,r)$ of $(\eps,r)$-F\o lner packings (see Introduction) in $G$. These packings $\cP$ are special equivalence relations on $V(G)$. If $x,y\in V(G)$ are in the same elements of $\cP$, then $x\equiv_{\cP} y$. The vertices that are not covered by the elements of $\cP$ form classes of size 1. Let us enumerate the vertices of $G$, $\{x_1,x_2,x_3,\dots\}$. Let the distance of two $(\eps,r)$-F\o lner packings $\cP_1$ and $\cP_2$ be $2^{-n}$ if 
\begin{itemize}
\item For $1\leq i,j \leq n-1$ we have that $x_i\equiv_{\cP_1} x_j$ if and only if $x_i\equiv_{\cP_2} x_j.$
\item For some $1\leq i \leq n-1$, $x_i\equiv_{\cP_1} x_n$ and $x_i\not\equiv_{\cP_2} x_n$ or
$x_i\not\equiv_{\cP_1} x_n$ and $x_i\equiv_{\cP_2} x_n$.
\end{itemize}
\noindent
It is not hard to see that $P_G(\eps,r)$ a compact space with respect to the above metric.

\noindent We call the graph $G$\textbf{ strongly F\o lner hyperfinite}
if for any $\eps>0$ there exist $r\geq 1$ and a Borel probability measure
$\nu$ on the compact space $P_G(\eps,r)$  such that for any vertex $x\in V(G)$
$$\nu(\{\cP\in P_G(\eps,r)\,\mid\, x\notin \tilde{\cP}\})<\eps\,,$$
\noindent
where $\tilde{\cP}$ denotes the set of vertices contained in the elements of the packing $\cP$.
\begin{proposition}\label{short1}
If the graph $G\in\grd$ is of Property A and is setwise F\o lner, then $G$ is strongly F\o lner hyperfinite.
\end{proposition}
\proof Fix $\eps>0$. Let $k>0$ be an integer
such that for any finite set
$L\subset V(G)$ there exists
an $\epsilon$-F\o lner set $H$ such that
$L\subset H \subset B_k(L)$. Also, let $R_{k+1}$ be the size of the largest $k+1$-ball in $V(G)$. By Theorem \ref{longcycle}, we have an integer $l>0$ such
that there exists a Borel probability measure $\mu$ on the compact space $\Sep(G,l)$ of $l$-separators such that for any $x\in V(G)$:
\begin{equation} \label{probbecs}
\mu(\{Y\,\mid\, x\in Y\})<\frac{\eps}{R_{k+1}}\,.
\end{equation}
\noindent
Then, for any $x\in V(G)$:
\begin{equation} \label{valosz}
\mu(\{Y\,\mid\, B_{k+1}(x)\cap Y \neq \emptyset\})<\eps.
\end{equation}
\noindent
We define the map $\Theta:\Sep(G,l)\to \Sep(G,l)$ by $\Theta(Y)=B_{k+1}(Y)$. Clearly, $\Theta$ is continuous. Now, let $Y$ be an $l$-separator
such that if we delete $Y$ together with all the incident edges, the remaining
components are $\{J^Y_i\}^\infty_{i=1}, |V(J^Y_i)|\leq l\,.$ 

\noindent
Then, if we delete $\Theta(Y)$ from $V(G)$ the remaining
vertices are in the disjoint union $\cup_{i=1}^\infty A^Y_i$, where
$A^Y_i=V(J^Y_i)\setminus B_{k+1}(Y)\,.$

\noindent
Now, for each $1\leq i <\infty$ we pick the smallest
$\eps$-F\o lner set $H_i$ such that $A^Y_i\subset H^Y_i \subset B_k(A^Y_i)\subset V(J^Y_i)$ in the following way. We enumerate the vertices of $V(G)$ and if there are more than
one $\eps$-F\o lner sets of the smallest size in $B_k(A^Y_i)$, then we pick
the first one in the lexicographic ordering. Note that for every $i$, $\diam_G(H_i)\leq r$, where $r=2k+l$.

\noindent
Therefore, we have a continuous map
$$\Phi:\Sep(G,l)\to P_G(\eps,r)\,.$$ 
If $B_{k+1}(x)\cap Y =\emptyset$ then $x$ is in some of the element of the packing $\Phi(Y)$.
Hence if $\nu$ is the push-forward measure $\Phi_\star(\Sep(G,l)),$ then by \eqref{valosz}
for any $x\in V(G)$ we have that
$$\nu(\{\cP\in P_G(\eps,r)\,\mid\, x\notin \tilde{\cP}\})<\eps\,.$$
\noindent
Hence, $G$ is strongly F\o lner hyperfinite. \qed
\vi
\begin{proposition} \label{short2}
If $G$ is strongly F\o lner hyperfinite then $G$ is fractionally almost finite as well.
\end{proposition}
\proof
Fix $\eps>0$. Then there exists $r\geq 1$ and a Borel probability measure $\mu$ of $P_G(\eps, r)$ such that
for every $x\in V(G)$
\begin{equation} \label{e31}
\mu(\{ \cP\,\mid y\in \tilde{\cP}\,\mbox{for all}\, y\,\mbox{ such that}\, d_G(x,y)\leq 1\})>1-\eps.
\end{equation}
For any $(\eps,r)$-F\o lner set $H$ let
$F(H)$ be defined as the $\mu$-measure of all $(\eps,r)$-F\o lner packings
$\cP$ such that $H\in\cP$. Then, $F$  clearly satisfies
the two conditions above. \qed
\vi
\begin{proposition} \label{short3}
If the graph $G\in \grd$ is fractionally almost finite then G is of F\o lner Property A.
\end{proposition}
\proof For $\eps>0$, let $r>0$ be
an integer and $F:\cF_G(\eps, r)\to \R$, $c_x$ be as in the definition of fractional hyperfiniteness.
Since $H$ is an $\eps$-F\o lner set, the
uniform probability measure $p_H$ defined on H is a $2d\eps$-F\o lner function.
Now, for $x\in V(G)$, let
$$P_x:=\sum_{x\in H, H\in \cF_G(\eps, r)} F(H) p_H+ c_x \delta_x\,.$$
Then, $P_x$ is a $2\eps d$-F\o lner function. If $y\sim x$ is an adjacent vertex then we have the following inequality.
$$\|P_x-P_y\|_1\leq
c_x+c_y +\sum _{x\in \partial(H), H\in \cF_G(\eps, r)} F(H) +
\sum _{y\in \partial(H), H\in \cF_G(\eps, r)} F(H)\leq 4\eps.$$
\noindent
Therefore, $G$ has F\o lner Property A. \qed
\vi
By Theorem \ref{prop1} and Proposition \ref{uniloc}, we immediately have the following result.
\begin{proposition}\label{short4}
If G is of F\o lner Property A, then $G$ is of Property A and it is setwise F\o lner.
\end{proposition}
\noindent
By Propositions \ref{short1}, \ref{short2}, \ref{short3} and \ref{short4} the
Short Cycle Theorem follows. \qed

\section{Strong F\o lner Hyperfiniteness implies Strong Almost Finiteness} \label{sec6}
\noindent In this section we establish one more equivalent of Strong F\o lner Hyperfiniteness.
First let us give a precise definition for strong almost finiteness. 
\begin{definition}
The graph $G\in\grd$ is strongly almost finite if for any $\eps>0$ there exists $r\geq 1$ and
a probability measure $\nu$ on the space $P_G(\eps,r)$ of $(\eps,r)$-F\o lner packings satisfying the following two conditions.
\begin{itemize}
\item $\nu$ is concentrated on tilings, that is, on packings $\cP$ that fully covers $V(G)$.
\item For each $x\in V(G)$ the $\nu$-measure of tilings such that $x$ in on the boundary
of the tile containg $x$ is less than $\eps$.
\end{itemize}
\end{definition}
\noindent
The goal of this section is to prove Theorem \ref{elsotetel}.
\proof 
The ``if'' part follows from the definition, we need to focus on the ``only if'' part. 
So, let $G\in\grd$ be a strongly F\o lner hyperfinite graph. The next proposition has analogues in Ornstein-Weiss theory, but the assumption of strong F\o lner hyperfiniteness in the present setting makes the proof simpler. 
\begin{proposition}\label{ow}
For any $\eps>0$ there exists $\delta>0$, $r\geq 1$ and an $(\eps,r)$-F\o lner packing
$\cP=\{H_i\}_{i=1}^\infty$ such that for any $\delta$-F\o lner set $T\subset V(G)$, the
subsets $H_i$ that are contained in $T$ cover at least $(1-\eps)|T|$ vertices of $T$.
\end{proposition}
\proof
Let $r\geq 1$ and $\nu$ be a Borel probability measure
on the space \\$P_G(\eps,r)$ of $(\eps,r)$-F\o lner packings such that
for any vertex $x\in V(G)$
$$\nu(\{\cP\,\mid x\notin \tilde{\cP}\})<\frac{\eps}{10}\,.$$
\noindent
Pick $\delta>0$ in such a way that if $T$ is a $\delta$-F\o lner set in
$V(G)$ then
$$\frac{|T'|}{|T|}>1-\frac{\eps}{10}\,,$$
\noindent
where
$$T':=\{y\in T\mid d_G(y,\partial(T))>2r\}\,.$$
\noindent
Observe that if $x\in T'$ for some $\delta$-F\o lner set $T$, then $B^G_{2r+1}(x)\subset T$. Hence, we have the following lemma.
\begin{lemma}\label{lemmaseged} If $H_1$ and $H_2$ are both $(\eps,r)$-F\o lner sets, $H_1$ intersects $T'$ and $H_2$ intersects the complement of $T$, then $H_1$ and $H_2$ are disjoint sets.
\end{lemma}
\noindent
Now, for $x\in T'$ let $\omega_x$ be a random variable on
the probability space $(P_G(\eps,r),\nu)$ such that
$\omega_x(P)=0$ if $x\in \tilde{P}$, $\omega_x(P)=1$ if $x\notin \tilde{P}$.
Then we have the following inequality for the expected value.
\begin{equation} \label{e51}
E(\omega_x)<\frac{\eps}{10}\,.
\end{equation}
\begin{lemma}
For any $\delta$-F\o lner set $T$
there exist an $(\eps,r)$-F\o lner packing $\cP$
such that
$|T''|>(1-\frac{\eps}{5})|T|$, where
$T''$ is the set of points in $T$ that are covered
by F\o lner-sets $H$ in the packing $\cP$ that
intersects $T'$.
\end{lemma}
\proof By \eqref{e51},
$$E(\sum_{x\in T'} \omega_x)<\frac{\eps}{10} |T'|\,.$$
\noindent
Therefore, there exists a packing $\cP\in P_G(\eps,r)$ that covers
at least $(1-\frac{\eps}{10})|T'|$ vertices in $T'$.
Hence,
\begin{equation} \label{e71}
|T''|> (1-\frac{\eps}{10}) (1-\frac{\eps}{10})|T|>(1-\frac{\eps}{5})|T|\,.
\end{equation}
\qed
\vi
Let us enumerate the $\delta$-F\o lner sets $\{S_1, S_2,\dots\}$ in $G$.
Let $\cP_n=\{H^n_i\}^\infty_{i=1}$ be an $(\eps,r)$-F\o lner packing such that
it covers the maximal amount of vertices in $\cup_{j=1}^n S_j$.
\begin{lemma}
For each $1\leq j\leq n$, the set
of $H_i^n$'s that are contained in $S_j$ covers at least $(1-\eps)|S_j|$ vertices
in $S_j$.
\end{lemma}
\proof Assume that
there exists $1\leq j \leq n$ that does not satisfy the covering statement of the lemma. Let $m$ be the number of vertices covered in $\cup_{q=1}^n S_j$ by $\cP_n$.
First, let us delete all the sets $H_i^n$ from the packing $\cP_n$ that are
in $S_j$. Now the number of vertices
covered in $\cup_{q=1}^n S_j$ remains at least $m-(1-\eps)|S_j|$.
Using the previous lemma we can add $(\eps,r)$-F\o lner sets in such a way that 
\begin{itemize}
\item We increase the number of vertices covered in $\cup_{j=1}^n S_j$ by at least \\
$(1-\frac{\eps}{5})|S_j|$.
\item We still obtain an $(\eps,r)$-F\o lner packing by Lemma \ref{lemmaseged}.
\end{itemize}
\noindent
So the new packing covers more than $m$ vertices in $\cup_{j=1}^n S_j$ leading to a contradiction. \qed
\vi
Now we can finish the proof of our proposition. Let $k$ be the maximal size of an $(\eps,r)$-F\o lner set in $G$. Let 
 $\{\cP_{n_i}\}_{i=1}^\infty$ be a convergent
subsequence in the compact space $P_G(\eps,r)$ converging to $\cP$. By the definition of convergence and the previous lemma, $\cP$ will satisfy the condition of our proposition. \qed
\begin{proposition} \label{sfhalmfin}
If the graph $G\in\grd$ is strongly F\o lner hyperfinite, then $G$ is almost finite. 
\end{proposition}
\proof
Fix $0<\eps<\frac{1}{3}$. Let $0<\delta<\frac{1}{2}$, $r>0$ so that by Proposition \ref{ow} there exists a
$(\epsilon,r)$-F\o lner packing $\cP=\{H_i\}^\infty_{i=1}$ 
so that for each $\delta$-F\o lner set $T$ at least
$(1-\eps)|T|$ vertices of $T$ are covered by some $H_i\subset T$.
\noindent
Since $G$
is fractionally almost finite by Theorem \ref{shortcycle}, there exists
$k>0$ and a non-negative function $F$ on the space $\cF_G(\delta,k)$ of $\delta$-F\o lner sets of radius less than $k$, such that
for any $x\in V(G)$ we have that
\begin{equation} \label{emasik}
\sum_{x\in T,\, T\in \cF_G(\delta,k)} F(T)+c_x=1\,,
\end{equation}
\noindent with $0\leq c_x <\delta\,.$
Pick a subset $K_i\subset H_i$ such that
$3\eps |H_i|<|K_i|<4\eps |H_i|.$
Let $A_1\subset V(G)$ be the set of vertices
not covered by any $H_i$ and let $A_2\subset V(G)$ be
the set of vertices that are in some $K_i$. That is,
for any  $\delta$-F\o lner set $T$ we have that
\begin{equation}\label{e73}
2|T\cap A_1|< 2\eps |T|<3\eps (1-\eps) |T|< |T\cap A_2|\,.
\end{equation}
\noindent
Let us construct a weighted, directed, bipartite graph $D(A_1,A_2)$ in the following way. 
For each $(\delta,k)$-F\o lner set $T$ so that $F(T)>0$ and for
each $x\in T\cap A_1$ we draw two outgoing edges towards $T\cap A_2$. One edge has weight $F(T)$, the other one has weight
$c_x\frac{F(T)}{1-c_x}$.
By \eqref{e73}, we can assume that for any $T$ the endpoints of
the drawn edges are different.
Also by \eqref{emasik}, for each vertex
$x\in A_1$, the sum of the weights on the outgoing edges is $1$ and for each vertex $y\in A_2$ the sum of the weights
on the incoming edges is less than or equal to $1$.

\noindent
Therefore our directed graph satisfies the Hall condition,
any finite subset $M$ of $A_1$ has at least $|M|$ adjacent
vertices in $A_2$. So by the Marriage Theorem, using a strategy somewhat similar to \cite{Down},
there exists an injective map $\Phi:A_1\to A_2$ such that
for any $x\in A_1$, $d_G(x,\Phi(x))<k$.
Now, for each $1\leq i<\infty$ set $S_i=H_i\cup \Phi^{-1}(H_i)$. Then,
$$|\partial(S_i)|\leq |\partial(H_i)|+| \Phi^{-1}(H_i)|\leq 5\eps |S_i|.$$
\noindent
Therefore, we have a partition $V(G)=\cup^\infty_{i=1} S_i$,
where each $S_i$ is a $5\eps$-F\o lner set and 
$$\diam_G(S_i)\leq 2k+r\,.$$
\noindent
Hence, G is almost finite. \qed.
\vi
Now let us finish the proof of our theorem. First fix $\eps>0$. Since $G$ is almost finite by Proposition \ref{sfhalmfin}, we have a partition
$V(G)=\cup_{i=1}^\infty T_i$, where all the $T_i$'s are $\eps$-F\o lner having diameter at most $t$. 
Let us pick $\delta>0$, $r>0$ and 
a probability measure $\nu$ on $\cP_G(\delta,r)$ in such a way that
\begin{itemize}
\item For any $(\delta,r)$-F\o lner set $F$ the set $F'$ is $\eps$-F\o lner whenever $F'$ is the union
of $F$ and some of the sets $T_i$ intersecting $F$.
\item The measure $\nu$ is concentrated on packings, where the
distance of two associated $(\delta,r)$-F\o lner sets is at least
$3t$.
\item For any $x\in V(G)$, $\nu(\{\cP\,\mid x\notin\tilde{\cP}\})<\delta.$
\end{itemize}
\noindent
For each packing $\cP=\{Q_j^{\cP}\}^\infty_{j=1}$ we construct a tiling
$\tau_\cP$ in the following way.
For $1\leq j < \infty$ let $R_j^{\cP}$ be the union of $Q_j^{\cP}$
and all the sets $T_i$ intersecting $Q_j^{\cP}$. Hence, by our condition $R_j^{\cP}$ is $\eps$-F\o lner. The remaining tiles
in $\tau_\cP$ are the sets $T_i$'s that are not intersecting any
$Q_j^{\cP}$. By pushing-forward $\nu$, we have a measure on the tilings $\tau_\cP$ satisfying the definition of strong almost finiteness. \qed

\section{Examples of strongly almost finite graphs} \label{sec7}
\noindent
In this section using the Short Cycle Theorem and Theorem \ref{elsotetel}, we extend the almost finiteness results of \cite{Down} about Cayley graphs to large classes of general graphs, to graphs of subexponential growth and to Schreier graphs of amenable groups.

\noindent
First let us recall the definition of subexponential growth.
\begin{definition}
The graph $G\in \grd$ is of subexponential growth if
$$\lim_{r\to \infty} \sup_{x\in V(G)} \frac{\ln (|B^G_r(x)|}{r}=0\,.$$
\end{definition}
\noindent
The following lemma is well-known, we provide the proof
for completeness. 
\begin{lemma} \label{subexp}
If $G\in \grd$ is a graph of subexponential growth, then for any $\eps>0$ there exists $r>0$ such that for all $x\in V(G)$ there is an $1\leq i\leq r$ so that 
$$\frac{|B^G_{i+1}(x)|}{|B^G_i(x)|}<1+\eps\,.$$.
\end{lemma}
\proof
Suppose that the statement of the lemma does not hold.
Then, there exists an $\eps>0$, a sequence of  vertices $\{x_n\in V(G)\}^\infty_{n=1}$ and an increasing sequence of natural numbers $\{i_n\}^\infty_{n=1}$ such that
for any $n\geq 1$, $\frac{\ln(|B^{G}_{i_n}(x_n)|}{i_n}\geq \ln(1+\eps)$, in contradiction with the definition of subexponentiality.\qed
\vi
\begin{proposition} \label{subexpclass}
Graphs $G\in\grd$ of subexponential growth are strongly almost finite.
\end{proposition}
\proof Let $G\in\grd$ be a graph of subexponential growth. By the Short Cycle Theorem and Theorem \ref{elsotetel}, it is enough to prove that $G$ is F\o lner and it is of Property A. Observe that being a F\o lner graph follows immediately from Lemma \ref{subexp}.
It has already been proved in \cite{Tu1} (Theorem 6.1) that graphs of subexponential growth are of Property A, nevertheless we give a very short proof of this fact using the Long Cycle Theorem.
Let $H\in\grd$ be the graph obtained by taking the disjoint union of all induced subgraphs of
$G$ up to isomorphism. Clearly, $H$ has subexponential growth, so $H$ is F\o lner. However, the the F\o lnerness of $H$ implies that $G$ is uniformly locally amenable, hence by the Long Cycle Theorem, $G$ is of Property A. \qed
\vi
Now $\Gamma$ be a finitely generated group
and $\Sigma$ be a finite, symmetric generating
set of $\Gamma$.
Let $H\subset\Gamma$ be a subgroup. Recall that the
Schreier graph $\Sch(\Gamma/H,\Sigma)$ is defined
as follows.
\begin{itemize}
\item The vertex set of $\Sch(\Gamma/H,\Sigma)$ consists of the the right cosets $\{Ha\}_{a\in \Gamma}$.
\item The coset $Ha$ is adjacent to the coset
$Hb\neq Ha$ if $Hb=Ha\sigma$ for some $\sigma\in\Sigma$. 
\end{itemize}
\noindent
If $H=e$ is the trivial subgroup, then the Cayley graph $\Cay(\Gamma,\Sigma)$ equals 
$\Sch(\Gamma/e,\Sigma)$.
\begin{proposition} \label{schrei} For any amenable group $\Gamma$ and symmetric generating system $\Sigma$, the graph
$\Sch(\Gamma/H,\Sigma)$ is strongly almost finite.
\end{proposition}
\proof 
Pick $\eps>0$. First, we will be working with $\Cay(\Gamma,\Sigma)$. Let $F$ be a $\frac{\eps}{2|\Sigma|}$-F\o lner set in $\Gamma$  containing the unit element.
By the amenability of $\Gamma$, such subset exists. 
For $x\in \Gamma$ let $P_{xF}$ be the
uniform probability measure  on the translate $xF$. By our condition on $F$, if $x\sim y$ we have that \begin{equation}\label{vege}
|xF\triangle yF|\leq |F\triangle (x^{-1}  yF\triangle yF)|+|yF\triangle F|<\eps |F|\,.
\end{equation}
\noindent
Therefore, $\|P_{xF}-P_{yF}\|_1<\eps$. Also
$\Supp (P_{xF})\subset B^{\Cay(\Gamma,\Sigma)}_{\diam(F)}(x)$.
Finally, we have that
$$\sum_{z\in \Gamma} \sum_{\sigma\in\Sigma}
|P_{xF}(z)-P_{xF}(z\sigma)|\leq
\frac{2|\partial (F)|}{|F|} |\Sigma|<\eps\,.$$
\noindent
That is, $P_{xF}$ is an $\eps$-F\o lner function.
By \eqref{vege}, $\Cay(\Gamma,\Sigma)$ is of F\o lner Property A. Net we will use
the functions $P_{xF}$ to prove that $\Sch(\Gamma/H,\Sigma)$  is of F\o lner Property A as well.

\noindent
Let $f:\Gamma\to \R$ be a finitely supported
non-negative function. Let us define $f^H:\Gamma/H\to \R$ by setting
$f^H(Hz)=\sum_{x\in Hz} f(x)\,.$
\begin{lemma} \label{char} We have that
\begin{itemize} 
\item[(a)] $\|f^H\|_1=\|f\|_1$.
\item[(b)] If $g:\Gamma\to \R$ is another   finitely supported
non-negative function then
$\|f^H-g^H\|_1\leq\|f-g\|_1$.
\item[(c)] 
$$\sum_{a\in \Gamma/H} \sum_{\sigma\in\Sigma}
|f^H(a)-f^H(a\sigma)|\leq 
\sum_{x\in \Gamma} \sum_{\sigma\in\Sigma}
|f(x)-f(x\sigma)|\,.$$
\end{itemize}
\end{lemma}
\proof
First, we have that
$$\|f^H\|_1=\sum_{a\in \Gamma/H} f^H(a)=
\sum_{a\in\Gamma/H}\sum_{x\in Ha} f(x)=\sum_{x\in\Gamma} f(x)=\|f\|_1\,.$$
\noindent
Then,
$$\|f^H-g^H\|_1=\sum_{a\in \Gamma/H} |f^H(a)-g^H(a)|=\sum_{a\in \Gamma/H} |(\sum_{x\in Ha} f(x))- (\sum_{x\in Ha} g(x))|\leq $$
$$\leq \sum_{x\in\Gamma}|f(x)-g(x)|=\|f-g\|_1\,.$$
\noindent
Finally,
$$\sum_{a\in \Gamma/H} \sum_{\sigma\in\Sigma}
|f^H(a)-f^H(a\sigma)|=
\sum_{a\in \Gamma/H} \sum_{\sigma\in\Sigma}
|(\sum_{x\in Ha} f(x))- (\sum_{x\in Ha} f(x\sigma))|\leq $$ $$\sum_{x\in \Gamma} \sum_{\sigma\in\Sigma}
|f(x)-f(x\sigma)|\,.\qed$$
\noindent
Now, we finish the proof of our proposition. Let $\Theta:\Gamma/H\to \Prob(\Gamma/H)$ be defined by
$\Theta(Hx):=P^H_{xF}$. Note that if $Hx=Hy$, then 
\begin{equation}\label{luck}
P^H_{xF}=P^H_{yF}\,,
\end{equation}
\noindent
so, $\Theta$ is well-defined.
Clearly, if $\Supp(f)\subset B_r^{\Cay(\Gamma,\Sigma)}(x)$, then 
$\Supp(f^H)\subset B_{r}^{\Sch(\Gamma/H,\Sigma)}(Hx)$. 
Therefore, for every $Hx\in \Gamma/H$ we have that $\Supp(\Theta(Hx))\subset  B_{r}^{\Sch(\Gamma/H,\Sigma)}(Hx)$, where
$F\subset B_r^{\Cay(\Gamma,\Sigma)}(e)$.

\noindent
By the previous lemma, for every $Hx\in \Gamma/H$, $\Theta(Hx)$ is an $\eps$-F\o lner probability measure. 
Again, by the previous lemma, if $Hy=H\sigma x$ we have
that
$$\|\Theta(Hx)-\Theta(Hy)\|_1<\eps\,.$$
\noindent
Therefore, $\Sch(\Gamma/H,\Sigma)$ is of F\o lner Property A.
Hence, by the Short Cycle Theorem and Theorem \ref{elsotetel}, our proposition follows. \qed
\begin{remark}
 Note that if $N$ is a normal subgroup in a free group $\Gamma$ and $\Gamma/N$ is a nonexact group then the Cayley graph of $\Gamma$ is of Property A, but the Cayley graph of $\Gamma/N$ is not of Property A. One might wonder, why Lemma \ref{char} does not imply that the Schreier graphs of Property A groups are of Property A themselves. The reason is that we used amenability in the proof of Proposition \ref{schrei} in a crucial way. The functions $\{P_{xF}\}_{x\in \Gamma}$ form an automorphism invariant system, that is why we have \eqref{luck}. If the group $\Gamma$ had such a canonical system of functions for every $\eps>0$, then all of the continuous actions of $\Gamma$ on the Cantor set would be topologically amenable (see Subsection \ref{topame}), hence the group $\Gamma$ would have to be amenable. Indeed, free continuous actions of nonamenable groups admitting invariant probability measures are never topologically amenable (see Proposition \ref{cstar}). Note that every countable group has free, minimal, continuous actions on the Cantor set that admit invariant probability measures \cite{elekfree}.
\end{remark}
\noindent
Let $\Gamma$ be an amenable group equipped with a generating system $\Sigma$, $H\subset \Gamma$ be a subgroup and let
$\pi_H:\Gamma\to \Gamma/H$ be the factor map, mapping $x$ into $Hx$. Then it is not true that
for any $\eps>0$ there is some $\delta>0$ such that the image of a $\delta$-F\o  lner set  is always an $\eps$-F\o lner set. Indeed, let $\Gamma=\Z\times \Z$ and $H$ be the first $\Z$-factor. Let $F_n=[0,n^2]\times [0,n]$. Now let $J_n$ be a set of $n$ elements in $\Gamma$ such that the second coordinates of these elements are positive integers greater than $n^2$ and their pairwise difference is at least 2. For large $n$ values both $G_n=F_n\cup J_n$ and $F_n$ are $\delta$-F\o lner sets with very small $\delta$, however, $\pi_H(G_n)$ is not even $\frac{1}{3}$-F\o lner set. By removing $J_n$ from $G_n$, we obtain a set such that its image is "very" F\o lner. The following proposition shows that this is always the case.
\begin{proposition}
Let $\Gamma$ be an amenable group with a symmetric generating set $\Sigma$. Then, for  any 
$\eps>0$ there is some $\delta=\delta(\eps)>0$ as in  Theorem \ref{prop1} such that 
if $H\subset \Gamma$ is a subgroup and $Q$ is a $\delta$-F\o lner set in $\Cay(\Gamma/H,\Sigma)$, then we have a subset $J$, $|J|<\eps |Q|$ so that the subset $\pi_H(Q\backslash J)$ is an $\eps$-F\o lner set in $\Sch(\Gamma/H,\Sigma)$.
\end{proposition}
\proof 
Fix $\eps>0$ and let $\delta=\delta(\eps)>0$ be as in Theorem \ref{prop1} so that if $p$ is a $\delta$-F\o lner probability measure  on the vertex set a graph $G\in \grd$, then there exists an $\eps$-F\o lner set $T$ inside the support of $p$ so that the $p$-measure of $\Supp (p)\backslash T$ is less than $\eps$.
Let $Q$ be a $\frac{\delta}{2d}$-F\o lner set in $\Cay(\Gamma,\Sigma)$. Then the uniform probability measure $p_Q$ is a $\delta$-F\o lner function. By Lemma \ref{char}, the function $Q^H$ is a $\delta$-F\o lner function supported on $\pi_H(Q)\subset \Sch(\Gamma/H,\Sigma)$. So by our condition, we have an $\eps$-F\o lner set $E$ inside $\pi_H(Q)$ such that the measure of $\pi_H(Q)\backslash E$ with respect to the probability measure $Q^H$ is less than $\eps$. That is,
$$|\pi^{-1}_H(\pi_H(Q)\backslash E)|\leq \eps |Q|\,.$$
\noindent
Therefore by choosing $J=\pi^{-1}_H(\pi_H(Q)\backslash E)$, our proposition follows. \qed

\section{Neighborhood convergence} \label{sec8} \label{neigh}

\begin{definition}
The graphs $G,H\in \grd$ are called \textbf{neighborhood equivalent}, $G\equiv H$ if for
any rooted ball $B^G_r(x)\subset G$ there exists a rooted ball $B^H_r(y)\subset H$ that is
rooted-isomorphic to $B^G_r(x)$ and conversely, for any rooted ball $B^H_r(u)\subset H$ there exists a rooted ball $B^G_r(v)\subset G$ that is
rooted-isomorphic to $B^H_r(u)$. So, if $\hat{G}$ denotes the set of all $r$-balls in $G$ up to rooted isomorphism, then $G$ and $H$ are neighborhood equivalent if and only if $\hat{G}=\hat{H}$.
\end{definition}
\noindent
We call a graph property $\cP\subset \grd$ a neighborhood equivalent property if for graphs $G\equiv H$, $G\in \cP$ if and only if $H\in\cP$.
\begin{proposition}\label{propconv1}
Amenability, Property A, being a F\o lner graph, almost \\ finiteness, $q$-colorability and having a perfect matching are all neighborhood equivalent properties.
\end{proposition}
\proof
Let us assume that $G$ is $q$-colorable for some $q\geq 2$ and $\alpha:V(G)\to \{1,2,\dots,q\}$ is a proper $q$-coloring. It is enough to prove that
any component of $H$ is $q$-colorable.
Let $x\in V(H)$ and for $n\geq 1$ $\beta_n:V(H)\to  \{1,2,\dots,q\}$ be labelings that are proper colorings restricted on
the ball $B^H_n(x)$. By neighborhood equivalence, such
labelings exist. Let $\beta_{n_k}\to \gamma$ be a
convergent subsequence. Then $\gamma$ is  proper $q$-coloring of the component of $H$ containing $x$. Similarly, we can prove that having perfect matching
or being almost finite is neighborhood equivalent, since
these properties can be described by colorings satisfying some local constraints.
It is straightforward to prove that amenability,
being a F\o lner graph and Property A (that is local hyperfiniteness) are neighborhood equivalent properties as well. \qed

\begin{definition} \label{ncon}
Let $B_1, B_2,\dots$ be an enumeration of the finite
rooted balls in $\grd$. We define a pseudo-metric on $\grd$ in the following
way. 
Let $\dist_{\grd}(G,H)=2^{-n}$ if
for $1\leq i \leq n-1$ $B_i\in (\hat{G}\cap \hat{H})$ or
$B_i\in (\hat{G}\cap \hat{H})^c$, and
$B_n\in \hat{G}\triangle \hat{H}$. It is easy to see that $\dist_\grd$ defines a
metric on the neighborhood equivalence
classes of $\grd$. So, a sequence
$\{G_n\}^\infty_{n=1}$ is a Cauchy-sequence in $\grd$ if for any rooted ball $B$, either $B\in \hat{G_n}$ for finitely many $n$'s or 
$B\in \hat{G_n}$ for all but finitely many $n$'s.
\end{definition}
\begin{proposition}\label{limitalmost}
The space of neighborhood equivalence is compact, or in other words, all Cauchy sequences are convergent.
\end{proposition}
\proof
First, let $\rgrd$ be the set of all rooted,
connected graphs of vertex degree bound d up to rooted isomorphisms. Again, we can define a metric
$\dist_\rgrd$ on $\rgrd$ by setting
$$\dist_\rgrd((G,x),(H,y))=2^{-n}\,,$$
\noindent
where $n$ is the largest integer for which the
rooted n-balls around x resp. y are rooted isomorphic. It is easy to see that $\rgrd$ is compact with respect to this metric.
Now, let $\{G_n\}_{n=1}^\infty\subset\grd$ be a Cauchy sequence.
Consider the set $\cA$ of all rooted graphs $(Q,x)$ that
are limits of sequences in the form of 
$\{G_n,x_n\}^\infty_{n=1}$, where 
$x_n\in V(G_n)$.
Clearly, if $(Q,x)\in \cA$, then all the rooted balls in $Q$ are rooted balls in all but finitely many $G_n$'s. On the other hand, if $B$ is a rooted ball in all but finitely many $G_n$'s then there exists $(Q,x)\in \cA$ so that
$B$ is a rooted ball in $Q$.
Therefore, if $\{Q_n\}^\infty_{n=1}$, is a countable dense subset of $\cA$, then for the graph $G$ having components $\{Q_n\}^\infty_{n=1}$ we have that $\lim_{n\to\infty} G_n= G$. \qed
\vi
We say that a countable set of graphs $\{G_n\}^\infty_{n=1}$ possesses the graph property $\cP$
if for the graph $B$ having components $\{G_n\}^\infty_{n=1}$, $B\in \cP$. 
The following proposition's proof is
similar to the one of Proposition \ref{propconv1} and left to the reader.
\begin{proposition} \label{seq} Let $\lim_{n\to\infty} G_n= G$. Then, if the set $\{G_n\}^\infty_{n=1}$ possesses
any of the properties listed in Proposition \ref{propconv1}, except amenability, so does $G$.
\end{proposition}
\noindent
By definition, all finite graphs are amenable, and limits of finite graphs can easily be non-amenable, e.g the 3-regular tree is non-amenable and it is the limit of large girth 3-regular graphs. However, we can define the amenability of a countable set of graphs in the following way.
\begin{definition}
The countable set of graphs $\{G_n\}^\infty_{n=1}$  is amenable if for any $\eps>0$ there exists $r\geq 1$ such that for any $n\geq 1$, the graph $G_n$ contains an $\eps$-Folner set of diameter at most $r$.
\end{definition}
\noindent
By the Long Cycle Theorem, any countable set of finite graphs having Property A is amenable. Obviously, this statement does not hold for infinite graphs.

\section{Hausdorff limits of graph spectra} \label{sec9}
\noindent
Let $G\in\grd$ be a finite or infinite graph and
$\cal{L}_G:l^2(V(G))\to l^2(V(G))$
\noindent
be the Laplacian operator on $G$ as in the Introduction.
\begin{proposition}
If $G$ and $H$ are neighborhood equivalent, then \\
$\Spec(\cl_G)=\Spec(\cl_H)$.
\end{proposition}
\proof
First, we need a lemma.
\begin{lemma} \label{spec1}
Let $P$ be a real polynomial, then
$\|P(\cl_G)\|=\|P(\cl_H)\|$.
\end{lemma}
\proof Fix some $\eps>0$.
Let $f\in l^2(V(G))$ such that $\|f\|=1$ and
$\|P(\cl_G)(f)\|\geq (1-\eps) \|P(\cl_G)\|$. We can assume that $f$ is supported
on a ball $B^G_s(x)$ for some $s>0$ and $x\in V(G)$. Let $t$ be the degree of $P$.
Then, $P(\cl_G)(f)$ is supported in the ball $B^G_{s+t}(x)$.
Since $G$ and $H$ are equivalent, there exists $y\in V(H)$ such that the ball
$B^G_{s+t}(x)$ is rooted-isomorphic to the ball $B^H_{s+t}(y)$ under some rooted-isomorphism $j$.
Then, $\|j_*(f)\|=1$ and $\|P(\cl_G)(f)\|= \|P(\cl_H)(j_*(f))\|$, where
$j_*(f)(z)=f(j^{-1}(z)),$ for $z\in B_s^G(x)$.
Therefore, $\|P(\cl_H)\|\geq (1-\eps)\|P(\cl_G)\|$ holds for any $\eps>0$. Consequently,
 $\|P(\cl_H)\|\geq \|P(\cl_G)\|$. Similarly,  $\|P(\cl_G)\|\geq \|P(\cl_H)\|$, thus our lemma follows. \qed
\vskip 0.1in
\noindent
By Functional Calculus, we have that
\begin{equation} \label{eqspec}
\|\phi(\cl_G)\|=\|\phi(\cl_H)\|
\end{equation}
\noindent
holds for any real continuous function $\phi$.
Observe that $\lambda\in\Spec(\cl_G)$ if and only if
for any $n\geq 1$ $\|\phi_n^\lambda(\cl_G)\|\neq 0$, where $\phi^\lambda_n$ is a piecewise linear, continuous, non-negative
function such that
\begin{itemize}
\item $\phi^\lambda_n(x)=1$ if $\lambda-\frac{1}{n}\leq x \leq \lambda + \frac{1}{n}\,,$
\item $\phi_n^\lambda(x)=0$ if $x\geq \lambda +\frac{2}{n}$ or $x\leq \lambda-\frac{2}{n}\,,$
\item and defined linearly otherwise.
\end{itemize}
\noindent
Therefore, by \eqref{eqspec} our proposition follows. \qed
\vi
The main goal of this section is to prove Theorem \ref{spectralconv}. 
\proof
The following lemma shows how to test whether
a certain value $\lambda$ is near to the
spectrum of the Laplacian.  
\begin{lemma} \label{test}
Fix $\eps>0$.
Let $\file$ be the following positive
continuous function on the real line.
\begin{itemize}
\item $\file(x)=0$ if $x\leq\lambda-\eps$ or
$x\geq \lambda+\eps\,.$
\item $\file(x)=1$ if $\lambda-\frac{\eps}{2}\leq x \leq \lambda +\frac{\eps}{2}\,.$
\item $\file$ is linear on the intervals
$[\lambda-\eps,\lambda-\frac{\eps}{2}]$ and $[\lambda+\frac{\eps}{2}, \lambda+\eps]$.
\end{itemize}
 Let $\pile$ be a real polynomial such that
$\sup_{x\in [0,2d]} |\file(x)- \pile(x)|\leq \eps$. \\ If $\|\pile(\cL_H)\|>\eps$ for
some $H\in\grd$, then there exists $\kappa\in \Spec(\cL_H)$ such that $|\kappa-\lambda|<\eps$.
\end{lemma}
\proof
By Functional Calculus, we have that
$$\|\file(\cL_H)\|\geq \|\pile(\cL_H)\|-\eps\,.$$ \noindent Therefore, $\|\file(\cL_H)\|>0$. Again by Functional Calculus, we can conclude that then there exists $\kappa\in \Spec(H)$ such that $|\kappa-\lambda|<\eps$. \qed
\vi
\begin{proposition}\label{pro32}
Let $\lim_{n\to\infty} G_n= G$ for
some convergent sequence \\ $\{G_n\}^\infty_{n=1}\subset \grd$. Suppose
that $\lambda\in \Spec(\cL_{G})\,.$
Then, for any $\eps>0$ there exists
$N_\eps>1$ such that if $n\geq N_\eps$
there exists $\lambda_n\in\Spec(\cL_{G_n})$ so
that $|\lambda_n-\lambda|\leq \eps$.
\end{proposition}
\proof
Let $\file$ and $\pile$ be as in Lemma \ref{test}.
By Functional Calculus, \\ $\|\file(\cL_G)\|=1$, so there exists
a function $f\in l^2(G)$, $\|f\|=1$ supported on
some ball $B^G_s(x)$ such that
\begin{equation}
\label{eq3}
\|\file(\cL_G)(f)\|>\eps.
\end{equation}
\noindent Let $m$ be the degree of $\pile$. Then $\pile(f)$
is supported on $B^G_{s+m}(x)$.
As in the proof of Lemma \ref{spec1}, we can see that if $\dist_\grd(G,H)$ is small
enough, then we have some $g\in l^2(H)$, $\|g\|=1$ supported on $B_s^H(y)$ such that
$$\|\pile(\cL_H)(g)\|=
\|\pile(\cL_G)(f)\|>\eps\,.$$
\noindent
Therefore $\|\pile(\cL_H)\|>\eps$, so our proposition follows from Lemma \ref{test}.
\qed
\vi
\begin{proposition} \label{prop33}
Let $\{G_n\}^\infty_{n=1}\subset \grd$ be a countable set
 of graphs of Property A converging to $G\in\grd$.
Suppose that for $0<\eps<\frac{1}{4}$ and $\lambda\geq 0$ there exists $N_\eps>0$ so that
if $n\geq N_\eps$, then
$$\Spec (\cL_{G_n})\cap (\lambda-\frac{\eps}{2}, \lambda+\frac{\eps}{2})\neq \emptyset\,.$$
\noindent
Then, $\Spec(\cL_G)\cap (\lambda-\eps,\lambda+\eps)\neq \emptyset$.
\end{proposition}
\proof First, fix $\eps>0$. Denote by $l$ the degree of the polynomial $\pile$. By Proposition \ref{seq}, the
graph $\tilde{G}\in\grd$ whose components
consist of $\{G_n\}^\infty_{n=1}$ and $G$ is of Property A. Therefore by the Long Cycle Theorem, there exists an integer $m$
 and a probability measure $\mu$ on $\Sep(\tilde{G},m)$ satisfying the following condition:
For all $x\in V(\tilde{G})$,
\begin{equation} \label{condi}
\mu(\{Y\in \Sep(\tilde{G},m)\mid 
x\in B^{\tilde{G}}_l(Y)\})<\delta\,,
\end{equation}
\noindent
 $1-\eps-6\sqrt[4]{\delta}>\eps$. 

\noindent
This condition can be fulfilled by the argument of the beginning of Proposition \ref{short1}.

\noindent
For $f\in l^2(\tilde{G}), \|f\|^2=1$
and $Y\in \Sep(\tilde{G},m)$ we define $f_Y$ by setting
\begin{itemize}
\item
$f_Y(x)=f(x)$ if $x\notin B^{\tilde{G}}_l(Y)$.
\item
$ f_Y=0$ otherwise.
\end{itemize}
\begin{lemma} \label{deltabecsles}
$$\mu(\{Y\in \Sep(\tilde{G},m)\mid 
 \|f_Y\|^2<1-\sqrt{\delta}\})<\delta\,.$$ \end{lemma}
\proof
By \eqref{condi}, we have that
$$\sum_{x\in V(\tilde{G})}\int_{\Sep(\tilde{G},m)}  f^2_Y(x) d\mu(Y) \geq (1-\delta)\,.$$
\noindent
So by the Monotone Convergence Theorem,
\begin{equation}\label{eq41}
\int_{\Sep(\tilde{G},m)}  \|f_Y\|^2\,d\mu(Y)\geq 1-\delta\,.
\end{equation}
\noindent
Let $A=\{Y\mid \|f_Y\|^2< 1-\sqrt{\delta}\}$. Then by \eqref{eq41}, we have that
$$\mu(A)(1-\sqrt{\delta})+1-\mu(A)\geq 1-\delta\,.$$
\noindent
Thus, $\sqrt{\delta}\geq \mu(A)\,.$ \qed
\vi
Define $$\|\pile(\cL_G)\|_{\diamondsuit}:=\sup_g \frac{\|\pile(\cL_G)g\|}{\|g\|}\,,$$
\noindent
where the supremum is taken for
all nonzero functions $g\in l^2(G)$ which are supported on $(B_l^{\tilde{G}}(Y))^c\cap H$ for some $Y\in\Sep(\tilde{G},m)$ and $H\subset V(G)$ is the vertex set of a component in $Y^c$. Note these functions do not form a vector space, so $\|\,\|_\diamondsuit$ is not a proper norm. 
Clearly, $\|\pile(\cL_G)\|_{\diamondsuit}\leq \|\pile(\cL_G)\|.$
Let
$$\|\pile(\cL_G)\|_{\square}:=\sup_g \frac{\|\pile(\cL_G)g\|}{\|g\|}\,,$$
\noindent
where the supremum is taken for 
all nonzero functions $g\in l^2(G)$  such that there exists $Y\in\Sep(\tilde{G},m)$ for which $g$ is supported on the complement of $B_l^{\tilde{G}}(Y)$.
\begin{lemma} \label{karo} 
$\|\pile(\cL_G)\|_{\diamondsuit}=
\|\pile(\cL_G)\|_{\square}$.
\end{lemma}
\proof
By definition, we have that  $\|\pile(\cL_G)\|_{\diamondsuit}\leq
\|\pile(\cL_G)\|_{\square}$.
Now let $Y\in \Sep(\tilde{G},m)$ and $g\in  l^2(G)$ such that $g$ is supported on $\cup_{n=1} (B_l^G(Y))^c\cap H_n)$, where $\{H_n\}^\infty_{n=1}$ is an enumeration of the elements of the component of the complement of $Y$. Let $g_n$ be the restriction of $g$ onto $(B_l^G(Y))^c\cap H_n$. Clearly, the functions $\{g_n\}^\infty_{n=1}$ are pairwise orthogonal. Since $l$ is the degree of $\pile$, the function $\pile(g_n)$ is supported on $H_n$. Indeed, the $l$-neigbourhood of $(B_l^G(Y))^c\cap H_n)$ is inside $H_n$.  Hence, the functions 
$\{\pile(g_n)\}^\infty_{n=1}$ are also pairwise orthogonal. Therefore,
$\|\pile(\cL_G)\|_{\diamondsuit} \geq 
\|\pile(\cL_G)\|_{\square}$. \qed
\vi
Similarly, we can define 
$\|\pile(\cL_{G_n})\|_{\diamondsuit}$  and $\|\pile(\cL_{G_n})\|_{\square}$.
Then, $\|\pile(\cL_{G_n})\|_{\diamondsuit}=\|\pile(\cL_{G_n})\|_{\square}\,.$
\begin{lemma} \label{szomb2}
\begin{equation} \label{egyenlog}
\|\pile(\cL_G)\|-\|\pile(\cL_G)\|_{\diamondsuit}< 3\sqrt[4]{\delta}\,.
\end{equation}
\end{lemma}
\proof
Let $f:V(G)\to \R$ be a function such that
$\|f\|=1$ and $\|\pile(\cL_G))f\|\geq \|\pile(\cL_G)\|-\sqrt[4]{\delta}\,.$
Let $f_Y$ be as above such that
$\|f_Y\|^2>(1-\sqrt{\delta})$. 
Observe that
$$1=\|f\|^2=\|f_Y\|^2+\|f-f_Y\|^2.$$
\noindent
Therefore, $\|f-f_Y\|\leq \sqrt[4]{\delta}\,.$
By the triangle inequality,
$$\|\pile(\cL_G) f_Y\|\geq
\|\pile(\cL_G) f\|-\|\pile(\cL_G) (f-f_Y)\|\,.$$
\noindent
Since $\sup_{0\leq t \leq 2d} |\pile(t)|\leq 1+\eps <2$ we have that
$$\|\pile(\cL_G) f_Y\|\geq
\|\pile(\cL_G) f\|-2\|f-f_Y\|\geq
\|\pile(\cL_G) f\|-2\sqrt[4]{\delta}\geq \|\pile(\cL_G)\|-3\sqrt[4]{\delta}\,.$$
\noindent
Since $\|f_Y\|\leq 1$ and $f_Y$ is supported on the union of the subsets  \\
$\{(B_l^{\tilde{G}}(Y))^c\cap H_n)\}^\infty_{n=1}$ we have that
$$\|\pile(\cL_G)\|_{\square}\geq \|\pile(\cL_G)\|-3\sqrt[4]{\delta}\,.
$$
\noindent
Thus, our lemma follows from Lemma \ref{karo}. \qed
\vi
Similarly, we have that 
\begin{equation} \label{egyenlogn}
\|\pile(\cL_{G_n})\|-\|\pile(\cL_{G_n})\|_{\diamondsuit}< 3\sqrt[4]{\delta}\,.
\end{equation}
\noindent
\begin{lemma} \label{normbecsles}
For large enough $n$, we have that
\begin{equation}
\|\pile(\cL_{G_n})\|-\|\pile(\cL_G)\|< 6\sqrt[4]{\delta}\,.
\end{equation}
\end{lemma}
\proof
By definition, $\|\pile(\cL_G)\|_\diamondsuit$ equals 
$\sup_g \frac{\|\pile(\cL_G)g\|}{\|g\|}\,$, where the supremum is taken for all $g$'s such that
$g$ is supported on $H\cap (B^{\tilde{G}}_l(H^c))$, where $H\subset V(G)$ is a set of diameter at most $m$ and its induced subgraph is connected. Indeed, $H^c$ is an $m$-separator. For these functions $g$, $\pile(\cL_G)g$ is supported on $H$.
Now, if $n$ is large enough the set of induced subgraphs (up to isometry) on such $H$'s are the same in $G_n$ and in $G$. Therefore, $\|\pile(\cL_{G_n})\|_\diamondsuit=\|\pile(\cL_G)\|_\diamondsuit.$ Hence, our lemma follows from the the inequalities \eqref{egyenlog} and
\eqref{egyenlogn}.\qed
\vi
By our assumption on the spectra for large enough $n$, $\|\file(\cL_{G_n})\|=1\,$. 
Hence, $\|\pile(\cL_{G_n})\|\geq 1-\eps$,
so $$\|\pile(\cL_G)\|\geq 1-\eps-6\sqrt[4]{\delta}\,.$$
\noindent
Thus, by our assumption on $\delta$ and by Lemma \ref{test}, our proposition follows.\qed 
\vi
Now we finish the proof of Theorem \ref{spectralconv}. Suppose that the compact sets
$\Spec(\cL_{G_n})$ do not converge to $\Spec(\cL_G)$ in the Hausdorff distance.

\noindent
\textbf{Case 1.  } There exists $\delta>0$, a sequence of positive integers $k_1<k_2<\dots$ and $\{\lambda_n\}^\infty_{n=1}\subset \Spec(\cL_G)$ such that 
$\inf_{\kappa\in \Spec(\cL_{G_{k_n}})}|\kappa-\lambda_n|> 2\delta\,.$ Let $\lambda\in \Spec (\cL_G)$ be a limit point of the sequence $\{\lambda_n\}^\infty_{n=1}.$ Then, we cannot have elements $\kappa_n$ in $\Spec(\cL_{G_{k_n}})$ such that for large enough $n$, $|\kappa_n-\lambda|<\delta$, in contradiction with Proposition \ref{pro32}.

\noindent
\textbf{Case 2.  } There exists $\delta>0$, a sequence of positive integers $k_1<k_2<\dots$ and $\kappa_n\in \Spec(\cL_{G_{k_n}})$ such that 
$\inf_{\lambda\in \Spec(\cL_G)}|\kappa_n-\lambda|> 2\delta\,.$
Let $\kappa$ be a limit point of the sequence $\{\kappa_n\}^\infty_{n=1}.$ 
Then, for large enough $n$ we have that
$$\Spec (\cL_{G_{k_n}})\cap (\kappa-\frac{\delta}{2}, \kappa+\frac{\delta}{2})\neq \emptyset\,.$$
\noindent
However, $\Spec(\cL_G)\cap (\kappa-\delta,\kappa+\delta)= \emptyset$, in contradiction with
Proposition \ref{prop33}. Therefore, our theorem follows. \qed
\vi
\begin{remark}
Let $\lim_{n\to\infty}G_n=G$, where $\{G_n\}^\infty_{n=1}$ is a large girth
sequence
of finite $3$-regular graphs and $G$ is the $3$-regular tree. Then for all $n\geq 1$,
$0\in\Spec(\cL_{G_n})$ and $0\notin\Spec(\cL_{G})$. 
Also, if $G$ is a large girth 3-regular expander graph, then its second smallest eigenvalue is away from zero. However, if $G$ is a large girth 3-regular graph containing an $\eps$-F\o lner set that is
smaller than $\frac{3}{4} |V(G)|$, then the second smallest eigenvalue of $G$ is very close to zero. That is, in general it is not true that the convergence of a finite graph sequence $\{G_n\}^\infty_{n=1}$
implies that $\{\Spec(\cL_{G_n})\}^\infty_{n=1}$ converges in the Hausdorff distance.
\end{remark}
\noindent
We finish this section with a purely combinatorial application of Theorem \ref{spectralconv}.
\begin{proposition}
For any positive integer $d$ and $\eps>0$ there exists $r>0$ so that if the families of rooted $r$-balls (up to rooted-isomorphism) in two finite planar graphs $G,H\in\grd$ coincide, then the Hausdorff distance of their spectra is at most $\eps>0$. 
\end{proposition}
\proof
A set of finite graphs $\cC\subset\grd$ is called monotone if it is closed under taking induced subgraphs. By the Large Cycle Theorem, a monotone subset $\cC$ is of Property A if and only if it is hyperfinite. The class of finite planar graphs (and the class of $H$-minor free graphs) are monotone hyperfinite (\cite{Benjamini}), so they are Property A as well. Assume that our proposition does not hold. Then there exists some $\eps>0$ and  a sequence
of pairs of planar graphs $\{(G_n,H_n)\}^\infty_{n=1}$ so that the families of the rooted $n$-balls in $G_n$ and $H_n$ coincide and
\begin{equation} \label{specut}
 \dist_{\Haus} (\Spec(\cL_{G_n}),\Spec(\cL_{H_n}))>\eps\,.
\end{equation}
\noindent
Let us pick a subsequence $\{G_{n_k}\}^\infty_{k=1}$ such that $\lim_{k\to \infty} G_{n_k}=G$ for some infinite graph $G$. By our conditions, we have that $\lim_{k\to \infty} H_{n_k}=G$.
So, by Theorem \ref{spectralconv} we have that
$\Spec(\cL_{G_n})\to \Spec(\cL_G)$ and $\Spec(\cL_{H_n})\to \Spec(\cL_G)$ in contradiction with \eqref{specut}. \qed

\section{Strongly almost finite graphs and classifiable $C^*$-algebras} \label{class}

\noindent
Arguably, one of the greatest achievements of operator algebras is the following result:

\noindent
\emph{"Separable, unital, simple, nuclear, infinite dimensional $C^*$-algebras with finite nuclear dimension that satisfy
the UCT are classified by their Elliott invariants."}

\noindent
The goal of this section is to show that starting from a minimal (see below), strongly almost finite graph $M\in\grd$ one can always construct tracial, classifiable $C^*$-algebras in a natural and almost canonical way. The construction is rather elementary and does not require any knowledge of $C^*$-algebras.

\noindent Somewhat similarly to 
the classification of finite simple groups, the theorem above 
was built on decades of work of $C^*$-algebrists that culminated
in the paper \cite{quasi}. A bit later it was proved \cite{nuclear}, that
for separable, unital, simple, nuclear $C^*$-algebras having finite nuclear dimension and
being $\cZ$-stable are equivalent.  Many of the known classifiable 
$C^*$-algebras that have traces are associated to minimal 
free actions of countable amenable groups on compact metric 
spaces. Note that for a very large class of $C^*$-algebras that are traceless the classifiability had been proved in \cite{Kirchberg} more than twenty years ago. It is conjectured that the so-called crossed product $C^*$-algebras of free minimal actions of countable
amenable groups on the Cantor set are always classifiable (see \cite{Down2} and \cite{Conley} for some interesting examples). 

\noindent
Our starting point is the following theorem:

\noindent
\emph{Let $\cG$ be a locally compact ample, minimal Cantor étale groupoid. Assume that $\cG$ is almost finite and topologically amenable. Then, the simple, unital $C^\star$-algebra $C^*_r(\cG)$ is $\cZ$-stable ( Corollary 9.11 of \cite{Mawu}, see also Corollary 5.8 of \cite{Cast}). By the results of \cite{Tu2} and \cite{nuclear} $C^*_r(\cG)$ satisfies the UCT-condition and has finite nuclear dimension. \textbf{Hence by Corollary D. of \cite{quasi}, $C^*_r(\cG)$
is classified by its Elliott invariants.}  }

\noindent
The statement above looks a bit frightening, but 
the reader should look at the bright side. Namely, the quoted result completely
eliminates $C^*$-algebras from the picture, by reducing the problem to groupoids.  
In constructing such groupoids, our strategy goes as follows.
\begin{itemize}
\item We introduce the notion of minimality for infinite graphs in a purely combinatorial fashion.
\item We explain the notion of \'etale Cantor groupoids (the ''unit space" of such a groupoid is the Cantor set) and show how very simple vertex and edge labeling constructions on a minimal graph $M\in\grd$ yield minimal \'etale Cantor groupoids.
\item We recall the notion of topological amenability and almost finiteness for \'etale Cantor groupoids and show that if the minimal graph $M$ is strongly almost finite, then  appropriate labelings of $M$ result in topologically amenable and almost finite \'etale Cantor groupoids.  Hence, Corollary 9.11  of \cite{Mawu} can be invoked to conclude that our new $C^*$-algebras are classifiable.  
\end{itemize}

\subsection{Minimal graphs} \label{minima}
\begin{definition}\label{minimal}
The connected graph $M\in\grd$ is\textbf{ minimal} if for every rooted ball $B=(B^M_r(x),x)$ in $M$
there exists a constant $r_B>0$ such that for all $y\in V(M)$, the ball $B^M_{r_B}(y)$ contains
a vertex $z$  so that the rooted $r$-ball around $z$ is rooted isomorphic to $B$.
\end{definition}
\noindent
Observe that a minimal graph is either F\o lner or nonamenable.
Clearly, the Cayley graphs of finitely generated groups are minimal. Nevertheless, minimal graphs can be very different from Cayley graphs.
\begin{itemize}
\item For any $\alpha>0$ real number there exists a minimal graph with asymptotic  volume growth $\alpha$ (see Section 5.2 of \cite{Elekurs}). On the other hand, the volume growth of a Cayley graph of polynomial growth is always an integer \cite{gromovpoly}.
\item A minimal graph might have $n$ ends, where $n$ is an arbitrary integer (see Section 10. in \cite{Elekurs} for a very general fractal-like construction). Infinite Cayley graphs have 
$1$, $2$ or infinitely many ends.
\item In fact, modifying somewhat the construction in \cite{Elekurs} one can construct a minimal
graph $M$ for any bounded degree graph $G$ by attaching suitable finite graphs to the vertices of $G$.
\end{itemize}
As we see soon, the notion of minimality is closely related to the concept of minimality for continuous group actions on compact spaces. Recall that the continuous action (that is, all elements are acting
by homeomorphisms) of a countable group $\Gamma$ on a compact metric space $X$, $\alpha:\Gamma\actson X$ is minimal if all the orbits of $\alpha$ are dense. Let $G\in\grd$ be a connected graph. Then, one can color the edges of $G$ with colors $\{c_1, c_2, \dots, c_{2d}\}$ in such a way that if the edges $e\neq f$ have the same color then they do not have a joint vertex.
This coloring defines the Schreier graph 
of an action of the $2d$-fold free product $\Gamma_{2d}$ of cyclic groups  of order $2$ generated by the elements 
$\{c_1, c_2, \dots, c_{2d}\}=\Sigma_{2d}$ (Section 5.1 \cite{Elekurs}). That is, every connected graph $G\in\grd$ is the underlying graph of a Schreier graph of $\Gamma_{2d}$.
Let $\rsgd$ denote the space of all rooted Schreier graphs of $\Gamma_{2d}$ with respect to the generating system $\Sigma_{2d}$ such that the underlying graph $G$ is in $\grd$. Similarly to the 
space $\rgrd$ defined in the proof of Proposition \ref{limitalmost}, $\rsgd$ is a compact metric space (Section 2.1 \cite{Elekurs}), 
also, $\rsgd$ is equipped with the natural, 
continuous (root-changing) action $\beta$ of the group  $\Gamma_{2d}$. 
\begin{lemma} \label{colmin}
Let $\cM$ be a closed, minimal $\Gamma_{2d}$-invariant subspace of $\rsgd$.
Then, for all elements $S$ of $\cM$ the underlying graph $M$ of $S$ is a minimal graph.
\end{lemma}
\proof
Let $(T,y)\in \cM$, where $y$ is the root. Suppose that the underlying graph $M$ of $(T,y)$ is not minimal. Then, there is a rooted ball $B$ in $M$ and there exist elements $\{g_n\in \Gamma_{2d}\}^\infty_{n=1}$ such that for $n\geq 1$ the rooted ball of radius $n$
around $g_n(y)$ does not contain balls that are rooted-isomorphic to $B$. Take a subsequence of the sequence $\{\beta(g_n)(T,y)\}^\infty_{n=1}$ converging to some element $(S,z)\in\rsgd$. Then, the underlying graph of $(S,z)$ does not contain a rooted ball isomorphic to $B$. Hence, the orbit closure of $(S,z)$ does not contain $(T,y)$, leading to a contradiction. \qed
\vi
It is not hard to see that the underlying graphs of the elements of $\cM$ are neighborhood equivalent to the graph $M$ and any connected graph that is neighborhood equivalent to $M$ is the underlying graph of some elements in $\cM$. 

\noindent
The idea above goes back to the paper of Glasner and Weiss \cite{GW}. They defined the \textbf{uniformly recurrent subgroups} (URS) of a countable group $\Gamma$ as the closed, minimal, invariant subspaces in $\Sub(\Gamma)$, the compact  space of subgroups of $\Gamma$.  If $H$ is an element of a URS, then the underlying graph of the
Schreier graph $\Sch(\Gamma/H,\Sigma)$ (with respect to a symmetric generating system $\Sigma$) is a minimal graph.
\noindent
 Conversely, let $M$ be a minimal graph and consider an arbitrary element $T$ of $\rsgd$ with $M$ as its underlying graph. Let $\cT$ be the orbit closure of $T$ and $\cM$ be a closed, minimal $\Gamma_{2d}$-invariant subspace in $\cT$. Then, $\cM$ is $\Gamma_{2d}$-isomorphic to a uniformly recurrent subgroup. Indeed, for each element of $\cM$ consider the stabilizer of the root.

We extend the definition of minimality for Cantor-labeled rooted $\Gamma_{2d}$-Schreier graphs with underlying graphs in $\grd$. A graph \( G \in \grd \) is said to be \emph{Cantor-labeled} if there exists a labeling function \( c : V(G) \to \mathcal{C} \) that assigns to each vertex of \( G \) a label from the Cantor set \( \mathcal{C} \). Denote the set of such graphs by
$\csgd$. Let us identify the Cantor set $\cC$ with the product set $\{0,1\}^\N$. If $c\in \cC$ and $c=\{a_1, a_2,\dots\}$ then we will refer to $a_i$ as the $i$-th coordinate of $c$. Similarly to $\rgrd$ and $\rsgd$ we have a natural compact metric structure on $\rcsd$. Again, the group $\Gamma_{2d}$ acts on 
$\rcsd$ by changing the root. For a rooted edge-colored, Cantor-labeled ball $B$, we denote by
$B_{(k)}$ the ball rooted-colored isomorphic to $B$ labeled with the finite set $\{0,1\}^k$ in such a way that for all vertex $x$ in $B_{(k)}$, the label of $x$ is $\pi_k(c)$, where $c$ is the Cantor label in $B$ and $\pi_k:\{0,1\}^\N\to \{0,1\}^k$ is the projection onto the first $k$ coordinates.
\begin{definition} 
The Cantor-labeled rooted Schreier graph $S\in \csgd$ is minimal if for all rooted labeled-colored ball $B=(B^M_r(x))$ in $S$
there exists a constant $r_B>0$ such that for all $y\in V(S)$, the ball $B^M_{r_B}(y)$ contains
a vertex $z$  so that $(B^S_r(z))_{(k)}$ is rooted labeled-colored isomorphic to $B_{(k)}$. Clearly, the underlying graph of a minimal graph $S\in\csgd$ is minimal in $\grd$.
\end{definition}
\noindent
The following lemma can be proved in the same way as Lemma \ref{colmin}.
\begin{lemma} \label{cantmin}
Let $\cM$ be a closed, minimal, $\Gamma_{2d}$-invariant subspace of $\csgd$.
Then, every element $S\in\cM$ is minimal. 
\end{lemma} \begin{lemma}\label{lemmauj38} Any
minimal graph $M\in\grd$ can be edge-colored properly and equipped with a Cantor-labeling in such a way that the resulting labeled-colored graph is minimal in $\csgd$.
\end{lemma}
\proof Let $\tilde{M}$ be an arbitrary element of $\csgd$ such that its underlying graph is isomorphic to $M$.  Let $\cL\subset \csgd$ be a closed, invariant, minimal subspace in the orbit closure of $\tilde{M}$. Let $x\in V(M)$. By the minimality of $M$, for any $n\geq 1$ there exists $S_n\in \cL$ such that 
$B_n^{S_n}(\mbox{root}\,(S_n))$ is rooted isomorphic to $B_n^M(x)$. Therefore, if $S$ is a limit point of the sequence $\{S_n\}^\infty_{n=1}$, the underlying graph of $S$ is isomorphic to $M$. Thus, by Lemma \ref{cantmin}, $S$ is minimal. \qed
\subsection{\'Etale Cantor groupoids and infinite graphs}
\noindent
Let $\alpha:\Gamma\actson \cC$ be a continuous action of a countable group $\Gamma$ on the
Cantor set $\cC$.
Assume that the action is\textbf{ stable}, that is, for every $g\in\Gamma$ there exists $r_g>0$
such that if $x\in\cC$ and $\alpha(g)(x)\neq x$ then $\dist_{\cC} (x,\alpha(g)(x))\geq r_g$. Here, $\dist_\cC$ is
a metric on $\cC$ that metrizes the compact topology. Note that an action is stable if and only if the stabilizer map $\Stab:\cC\to \Sub(\Gamma)$ (the space of subgroups of $\Gamma$, \cite{GW}) is continuous. Of  course, free continuous actions are always stable.

\noindent
Now we define an equivalence relation on $\cC$ and the groupoid of $\alpha$. We set $x\equiv_\alpha y$ if for some $g\in\Gamma$ $\alpha(g)(x)=y$. The  groupoid \cite{Sims}
$\cG_\alpha\subset \cC\times \cC$ is the set of pairs $(x,y)$ such that $x\equiv_\alpha y$. The product is given by $(x,y)(y,z)=(x,z)$. The range map is given by $r:(x,y)\to y$ and the source map is given by $s:(x,y)\to x$. The inverse map $\gamma$ is given by $\gamma:(x,y)=(y,x).$ The \textbf{unit space} of the groupoid $\cG^0_\alpha$ is the set of pairs $(x,x), x\in \cC$.   Then
by the stability of the action $\alpha$,
\begin{itemize}
\item for every $(x,y)\in \cG_\alpha$ such that $\alpha(g)(x)=y$ for some $g\in \Gamma$, there exist clopen sets $U,V$ in the Cantor set, $x\in U$ and $y\in V$, such that
$\alpha(g):U\to V$ is a homeomorphism,
\item and if also $\alpha(h)(x)=y$, then there exists
a clopen set $x\in W\subset \cC$ such that
$$\alpha(g)\mid_W=\alpha(h)\mid_W\,.$$
\end{itemize}
\noindent
The topology on $\cG_\alpha$ is defined in the following way.
The base neighbourhoods of the element $(x,y)$ are in
the form of $(U,V,g,x,y)$, where $\alpha(g):U\to V$ as above,
and $(a,b)\in (U,V,g,x,y)$ if $a\in U$ and $\alpha(g)(a)=b$. Now, we can easily check that
$r:\cG^0_\alpha\to \cC$ is in fact a homeomorphism and for any pair $(x,y)$ such that $\alpha(g)(x)=y$, $r:(U,g,x,y)\to U$ is a homeomorphism, that is, $r$ is a \textbf{local homeomorphism.} Consequently, $\cG_\alpha$ is a locally compact Hausdorff \textbf{\'etale groupoid}  with  unit space isomorphic to the Cantor set (\cite{Sims}), as required in Corollary 5.8 in \cite{Cast}.
The \'etale groupoid $\cG_\alpha$ is called minimal if the action $\alpha$ is minimal \cite{Sims}. 

\noindent
Now, let us consider the $\Gamma_{2d}$-action $\beta_d:\Gamma_{2d}\actson \csgd$.
Recall that we connect $x\neq y\in \csgd$ with an edge if for some generator $c_i$ we have $\Gamma_{2d}(c_i)(x)=y$. For $x\in \csgd$, the\textbf{ orbit graph} of $x$ is the connected component of the graph above containing $x$, and the \textbf{rooted orbit graph }of $x$ is the orbit graph of $x$ with root $x$.
We have two minor problems to settle.
\begin{itemize}
\item The action $\beta_d$ is not stable. Consider the rooted Cayley graph $K$ of $\Gamma_{2d}$ generated by the elements $c_1,c_2,\dots, c_{2d}$ such that all vertices of $K$ are labeled with the same element $t\in \cC$. Then, $K$ is an element of $\csgd$ fixed by all elements of $\Gamma_{2d}$. In every neighborhood $U$ of $K$ there are graphs that are not fixed by any element of $\Gamma_{2d}$.
\item  We might expect that the rooted orbit graph of an element $N\in \csgd$ is rooted isomorphic to the underlying rooted graph of $N$.  Unfortunately, the rooted orbit graph of the above $K$ is a graph of no edges and has one single vertex. 
\end{itemize}
\noindent
We will put more restrictions on the Cantor labelings to get rid of these inconveniences. Take an arbitrary rooted Schreier graph $G\in\rsgd$ and label the vertices of
the underlying rooted graph by elements of $\cC$ in the following way: For any $n\geq 1$ there exists an $\eps>0$ such that if $0< d_G(x,y)\leq n$,  we have $\dist_\cC(l(x),l(y))\geq \eps$. Here, $l$ is the labeling function.
These \emph{proper} labelings exist and the action of $\Gamma_{2d}$ on the orbit closure of the labeled version of $G$ is stable. Also, if $K$ is an element of the orbit closure then the rooted orbit graph of $K$ is rooted isomorphic to $K$ (see \cite{Elekurs}, \cite{elekfree} for details). So, let us start  with a minimal graph $M\in\grd$ and equip it with a proper Cantor vertex labeling and a proper $\Sigma_{2d}$-edge labeling to obtain $\tilde{M}\in \csgd$. Let $E\subset \csgd$ be a minimal, closed invariant subspace in the orbit closure of $\tilde{M}$.  Then, $E$ is a stable, minimal, closed invariant subspace $E$ of $\csgd$ such that each element of $E$ is rooted isomorphic to its own rooted orbit graph and the underlying graphs are neighborhood equivalent to $M$. Since the action on $E$ is minimal, $E$ cannot have isolated points, so $E$ is homeomorphic to $\cC$. If $\beta_E$ is the restriction of $\beta_d$ on such a space $E$, then we call $\cG_{\beta_E}$  the minimal \'etale Cantor groupoid associated to $E$.
\subsection{Topologically amenable and almost finite \'etale groupoids.}
\noindent
Graph properties such as almost finiteness, finite asymptotic dimension, paradoxicality or Property A
have been defined for free Cantor actions of finitely generated groups. Informally speaking, the meaning in this context is that the property holds for every orbit in a continuous way.
This idea has already been extended to \'etale Cantor groupoids as well.
 
\noindent
\textbf{The continuous version of Property A} is called \textbf{topological amenability} for historical reasons. It was introduced in \cite{Anant} more than twenty years ago. In the definition of Property A we have probability measures
concentrated on $r$-balls around vertices for certain values $r$. First note that for a fixed $r$ and a fixed set of colors there exist only finitely many rooted edge-colored $r$-balls (up to rooted colored isomorphisms) in graphs in $\rsgd$. Of course, for each such ball $B$ there exist uncountably many probability measures supported on the vertices of $B$. Nevertheless, for any $\eps>0$ there exists a
finite set of probability measures $\{p^{B,\eps}_i\}^{N_{(B,\eps)}}_{i=1}$ supported on $B$, such that for every probability measure $p$ supported on $B$
there exists $1\leq i \leq N_{(B,\eps)}$ such that $\|p-p^{B,\eps}_i\|_1<\frac{\eps}{2}$. Hence, we can suppose that in the definition of Property A, the functions $\Theta(x)$ are all in the form of $p^{B,\eps}_i$, for some rooted edge-colored ball $B$ and $\eps>0$. So, we have the following simple version of topological amenability in the case
of our groupoids $\cG_{\beta_E}$.
\begin{definition}[\textbf{Topological amenability}]\label{topame} The groupoid $\cG_{\beta_E}$ is topologically amenable  if for any $\eps>0$
there exists $r>0$ and a partition $\cP$ of the totally disconnected space $E$ into finitely many clopen sets $\{U_\kappa\}_{\kappa\in P_{C^r_{\eps}}}$, and for each $x\in E$ there exists a probability measure $p_x$ on the orbit graph of $x$ supported on the $r$-ball around $x$ such that
\begin{itemize}
\item if $x\in U_\kappa$, then the probability measure $p_x$ has type $\kappa$, for some $\kappa\in P^{r}_{C_{\eps}}.$ Here $P^{r}_{C_{\eps}}$ is the finite set of all probability measures in the form of $p^{B,\eps}_i$, where $B$ is a rooted edge-colored $r$-ball. We call the elements of $P^{r}_{C_{\eps}}$ ''types".
\item For any $x\in E$ and generator $c_j\in \Sigma_{2d}$, $\|p_x-p_{\beta(c_j)(x)}\|_1<\eps$.
\end{itemize}
\end{definition}
\noindent
Similar definition can be given for any stable action of a finitely generated group with a given symmetric generating set.
\begin{proposition} \label{propagro}
For every minimal  graph $M\in \grd$ of Property A, we have an invariant subspace $E$ as above, so that the associated \'etale groupoid $\cG_{\beta_E}$ is
topologically amenable.
\end{proposition}
\proof Since $M$ is of Property A, there exists $r_n$ and for each $x\in V(M)$ a probability measure $\Theta(x)$ of
type in $P^{r_n}_{C_{\eps}}$ such that if $x$ and $y$ are adjacent vertices, then $\|\Theta(x)-\Theta(y)\|<\eps.$ Now let us denote by $Q_n$ the set $P^{r_n}_{C_{\eps}}$. For each $n\geq 1$ we have a vertex labeling by $Q_n$ of $V(M)$, where
the label of $x$ is the type of $\Theta(x)$. Altogether, we have a labeling of $V(M)$ by the product set $\prod_{n=1}^\infty Q_n$ that is isomorphic to the Cantor set. Also, we add edge-colors and a Cantor labeling in the way described above to ensure stability. So, we have a $\cC\times \cC$-labeling of $V(M)$, however, $\cC\times \cC$ is still homeomorphic to $\cC$. Add an arbitrary root to the resulting graph to obtain a rooted colored-labeled graph $S\in\csgd$. 
Now, find a closed, minimal invariant subspace $E$ in its orbit closure in $\csgd$. Since the $\prod_{n=1}^\infty Q_n$-labelings encode the witnessing of Property A for $M$, it also witnesses Property A for any other graph in the orbit closure, that is $E$ satisfies the conditions of Definition \ref{topame}, so $E$ is topologically amenable. \qed
\vi
 The notion of \textbf{almost finiteness for \'etale Cantor groupoids}  was defined by Matui (Definition 6.2 \cite{Matui}). We will need this concept only for the case of \'etale groupoids associated to stable actions of $\Gamma_{2d}$. One can check that the following simple definition that is analogous to Definition \ref{topame} applies to the case of our \'etale groupoids $E$, hence they are almost finite in the sense of Matui. 
\begin{definition} [\textbf{Almost finiteness for groupoids}]\label{d2f26}
The groupoid $\cG_{\beta_E}$ is almost finite if for any $\eps>0$
there exists $r>0$, $K_\eps$ and a partition  of the totally disconnected space $E$ into finitely many clopen sets $\{W_i\}_{i=1}^n$ such that
\begin{enumerate}
\item if $x,y\in E$, $\alpha(g)(x)=y$ for some $g\in\Gamma$,
and $x,y$ are in the same clopen set, then either  
$d_E(x,y)\leq K_\eps$ or $d_E(x,y)\geq 3K_\eps$, where $d_E$ is the graph distance on the orbit graphs.
\item For any $x\in E$, the set $H_x$ is $\eps$-F\o lner, where
$$H_x=\{z\in E\,,x,z\,\mbox{are in the same clopen set}\,W_k\,\mbox{and}\, d_E(x,z)\leq K_\eps\,\}.$$ 
\end{enumerate}
\end{definition}
\noindent
One should note that a definition of almost finiteness was given by Kerr \cite{Kerr} for free actions of amenable groups on compact metric spaces. If the amenable group $\Gamma$ acts on the Cantor set freely then the
almost finiteness of the associated étale groupoid in the sense of Matui is equivalent to the almost finiteness of the free action in the sense of Kerr. The definition of Kerr was extended to non-free actions by Joseph \cite{Joseph}. It is important to note that for such non-free actions the almost finiteness of the associated étale groupoid in the sense of Matui does not necessarily imply the almost finiteness of the action in the sense of Joseph. 
 
\noindent
The following proposition can be proved in the same way as Proposition \ref{propagro}.
\begin{proposition} \label{almagro}
For any minimal almost finite graph $M\in \grd$, we have an invariant subspace $E$ as above, so that the associated \'etale groupoid $\cG_{\beta_E}$ is
almost finite.
\end{proposition}
\noindent
\begin{corollary} \label{maincorol}
For every minimal strongly almost finite graph $M\in \grd$ we have an invariant subspace $E$ as above, so that the associated \'etale groupoid $\cG_{\beta_E}$ is both topologically amenable and almost finite.
\end{corollary}
\proof  By Theorem \ref{shortcycle} and Theorem \ref{elsotetel}, the graph $M$ is Property A and almost finite. Label the vertices of $M$ by the product of the labelings in Proposition \ref{propagro} (that encodes Property A) and the labelings in Proposition \ref{almagro} (that encodes almost finiteness). Let $S$ be the new labeled graph. Now, find a minimal, closed invariant subspace $E$ in the orbit closure of $S$.  Putting together Proposition \ref{propagro} and Proposition \ref{almagro}, we obtain the corollary. \qed 
\begin{question}
Is it true that a minimal étale groupoid is almost finite
if all of its orbit graphs are strongly almost finite?
\end{question}
\noindent
Note that if the answer for this question is yes, then
the groupoids associated to stable Cantor actions of amenable groups are always almost finite.
\begin{remark}
We could start with any finitely generated amenable group $\Gamma$ (with some finite generating system $\Sigma$) and an element $H$ of an arbitrary uniformly recurrent subgroup of $\Gamma$. By Proposition \ref{schrei}, the underlying graph of $\Sch(\Gamma/H,\Sigma)$ is a strongly almost finite minimal graph. Using this Schreier graph, we can repeat the construction above to obtain a stable, minimal, $\Gamma$-action $\alpha$ such that all orbit graphs are 
isomorphic to $S$ and the associated \'etale groupoid is both topologically amenable and almost finite.
\end{remark}

\noindent
Now we are ready to prove the main result of this section.
\begin{proof}[Proof of Theorem \ref{elliott}]
The algebras associated to the $E$'s in Corollary \ref{maincorol} are always tracial (see Section 9 in  \cite{Elekurs}) due to the existence of
invariant probability measures on $E$. The existence of such invariant measures follows from the amenability of
the orbit graphs of $E$ (a consequence of $M$ being almost finite) in the same way as one proves that continuous actions of amenable groups on compact metric spaces always admit invariant probability measures (Theorem 6 \cite{Elekurs}). 
Hence, by Corollary 9.11 in \cite{Mawu} cited at the beginning of the section, we finish the proof of our theorem. \end{proof}

\noindent
We also obtain a dynamical characterization of strong almost finiteness in the case of minimal graphs.
\begin{proposition}\label{cstar}
A minimal graph $M\in \grd$ is strongly almost finite if and only if there exists a stable action $\alpha:\Gamma\actson\cC$, for a finitely generated group $\Gamma$ with a finite generating set $\Sigma$ with the following properties.
\begin{itemize}
\item All the orbit graphs of $\alpha$ are neighborhood equivalent to $M$.
\item The action $\alpha$ is topologically amenable, admitting an invariant probability measure. 
\end{itemize}
\end{proposition}
\proof As we have seen in the proof of Theorem \ref{elliott}, if $M$ is  strongly almost finite, actions as above always exist.
Now, assume that $M$ is not strongly almost finite. 
If $M$ is not of Property A, then the required action cannot be topologically amenable. If $M$ is of Property A, but not strongly almost finite, then by Theorem \ref{shortcycle} and Theorem \ref{elsotetel}, $M$ is not F\o lner. Hence by minimality, $M$ must be nonamenable. If the action were topologically amenable, the action would be hyperfinite with respect to any invariant probability measure $\mu$, and $\mu$-almost all of its orbit graphs would be amenable  \cite{Kechris}. This leads to a contradiction. \qed

\noindent
\emph{G\'abor Elek}, Department of Mathematics and Statistics, Lancaster University, Lancaster, United Kingdom and Alfr\'ed R\'enyi Institute of Mathematics, Budapest, Hungary.

\noindent
\texttt{g.elek[at]lancaster.ac.uk}
\vi
\emph{\'Ad\'am Tim\'ar}, Division of Mathematics, The Science Institute, University of Iceland, Reykjavik, Iceland
and
Alfr\'ed R\'enyi Institute of Mathematics, Budapest, Hungary.

\noindent
\texttt{madaramit[at]gmail.com}\\ 

\begin{thebibliography}{99}
\bibitem{Abert} {\sc M. Ab\'ert, A. Thom and B. Vir\'ag}, Benjamini-Schramm convergence and pointwise convergence
of the spectral measure. (preprint at http://www.math.uni-leipzig.de/MI/thom)
\bibitem{Anant}  {\sc C. Anantharaman-Delaroche and J. Renault}, Amenable groupoids, Foreword by Georges Skandalis and Appendix B by E. Germain, L’Enseignement Math\'ematique, Geneva, \textbf{196 }(2000)
\bibitem{Ara} {\sc P. Ara, C. Bönicke, Christian, J. Bosa and K. Li},
The type semigroup, comparison, and almost finiteness for ample groupoids. {Ergodic Theory Dynam. Systems} \textbf{43}(2023), no.2, 361--400.
\bibitem{Bartholdi} {\sc L. Bartholdi and R. I. Grigorchuk}, On the spectrum of Hecke type operators related to some fractal groups.{\sl  Tr.
Mat. Inst. Steklova} \textbf{231} (2000), no. Din. Sist., Avtom. i Beskon. Gruppy, 5--45.
\bibitem{Benjamini} {\sc I. Benjamini, O. Schramm and A. Shapira}, Every minor-closed property of sparse graphs is testable.
{\sl Adv. Math.}\textbf{ 223} (2010), no. 6, 2200--2218.
\bibitem{Block} {\sc J. Block and S.  Weinberger} Aperiodic tilings, positive scalar curvature, and amenability of spaces. {\sl Journal of the American Mathematical Society}, (1992) 5(4), 907-918.
\bibitem{Brodzki} {\sc J. Brodzki, G. Niblo, J. \v{S}pakula, R. Willett and N. Wright}, Uniform local amenability. {\em J. Noncommut. Geom.} \textbf{7} (2013), no. 2, 583--603. 
\bibitem{nuclear} {\sc J. Castillejos, S. Evington, A. Tikuisis, S. White andd W. Winter}, Nuclear dimension of simple $C^*$ -algebras.
{\sl Invent. Math.}\textbf{ 224} (2021), no. 1, 245--290.
\bibitem{Cast} {\sc J. Castillejos, K. Li and G. Szab\'o}. On tracial $\cZ$-stability of \\ simple non-unital $C^*$-algebras. Canadian Journal of Math (2023)  https://doi.org/10.4153/S0008414X23000202
\bibitem{Ceccherini} {\sc T.G. Ceccherini-Silberstein, R.I. Grigorchuk and P. de la Harpe}, Amenability and paradoxical decompositions
for pseudogroups and for discrete metric spaces. {\sl Proc. Steklov Inst. Math.}\textbf{ 224 }(1999), 57--97.
\bibitem{Chen} {\sc X. Chen, R. Tessera, X. Wang and G. Yu}, Metric sparsification and operator norm localization.
Adv. Math.\textbf{ 218} (2008), no. 5, 1496--1511.
\bibitem{Conley}{\sc C. Conley, S. Jackson, A. Marks, B. Seward, R. Tucker-Drob},
Borel asymptotic dimension and hyperfinite equivalence relations.
{\sl to appear in Duke Math. J.}
\bibitem{Connes} {\sc A. Connes}, Classification of injective factors. Cases $II_1, II_\infty, III_\lambda, 1.$ Ann. of Math. (2) \textbf{104}
(1976), 73–115.
\bibitem{CFW} {\sc A.Connes, J. Feldman, J. and B. Weiss,} An amenable equivalence relation is generated by a single transformation. {\sl Ergodic theory and dynamical systems}, (1981) 1(4), 431-450.
\bibitem{Dodziuk}{\sc J. Dodziuk}, Difference equations, isoperimetric inequality and transience of certain random
walks. {\sl Trans. Amer. Math. Soc.}\textbf{ 284} (1984), 787--794.
\bibitem{Down} {\sc T. Downarowicz, D. Huczek and G. Zhang}, Tilings of amenable groups.
{\sl J. Reine Angew. Math.}\textbf{ 747} (2019), 277--298.
\bibitem{Down2} {\sc T.Downarowicz and G. Zhang},
Symbolic Extensions of Amenable Group Actions and the Comparison Property. {\sl Memoirs of the Amer. Math. Soc.} \textbf{1390} (2023)
\bibitem{Dye} {\sc H. Dye}, On groups of measure preserving transformations. I. {\sl Amer. J. Math.} \textbf{81} (1959),
119-159; II, ibid., \textbf{85} (1963), 551-576.
\bibitem{Elek1} {\sc G. Elek}, The K-theory of Gromov’s translation algebras and the amenability of discrete groups. {\sl Proc. Amer.
Math. Soc.} \textbf{125 }(1997), 2551--2553.
\bibitem{Elekcost} {\sc G. Elek}, The combinatorial cost. {\sl Enseign. Math.}, \textbf{53}, (2007), 225–235.
\bibitem{Elektimar} {\sc G. Elek and A. Tim\'ar}, Quasi-invariant means and Zimmer amenability.
(preprint)  https://arxiv.org/abs/1109.5863 
\bibitem{Elekurs} {\sc G. Elek}, Uniformly recurrent subgroups and simple 
$C^*$-algebras. {\sl J. Funct. Anal.} {\bf 274} (2018), no. 6, 1657--1689. 
\bibitem{Elekquali}{\sc G. Elek}, Qualitative graph limit theory. Cantor Dynamical Systems and Constant-Time Distributed Algorithms. (preprint)
https://arxiv.org/abs/1812.07511
\bibitem{Elekula} {\sc G. Elek}, Uniform local amenability implies property A.
{\sl Proc. Amer. Math. Soc.}\textbf{ 149} (2021), no. 6, 2573--2577.
\bibitem{elekfree} {\sc G. Elek}, Free minimal actions of countable groups with invariant probability measures.
{\sl Ergodic Theory Dynam. Systems}\textbf{ 41} (2021), no. 5, 1369--1389.
\bibitem{Eleklocal} {\sc G. Elek}, Planarity can be verified by an approximate proof labeling scheme in constant-time.
{\sl J. Combin. Theory Ser. A}\textbf{ 191 }(2022), Paper No. 105643, 17 pp.
\bibitem{uba} {\sc G. Elek and A. Tim\'ar}, Uniform Borel amenability is equivalent to randomized hyperfiniteness, https://arxiv.org/abs/2408.12565
\bibitem{Folner} {\sc E. F\o lner}, On groups with full Banach mean value. {\sl Math. Scand.} \textbf{3} (1955), 
243--254.
\bibitem{GW} {\sc E. Glasner and B. Weiss}, Uniformly recurrent subgroups. {\sl Recent trends in ergodic theory and dynamical systems}, 
63--75, Contemp. Math., {\bf 631}, Amer. Math. Soc., Providence, RI,( 2015).
\bibitem{Grigorchuk} {\sc R. Grigorchuk, T. Nagnibeda and A. P\'erez} On spectra and spectral measures of Schreier and Cayley graphs.
{\sl Int. Math. Res. Not.} (2022), no. 15, 11957--12002.
\bibitem{gromovpoly} {\sc M. Gromov}, Groups of polynomial growth and expanding
maps. {\sl Publ. Math. IHES} , \textbf{53}(1) (1981) 53–78.
\bibitem{Gromov} {\sc M. Gromov}, Random walk in random groups. 
{\sl Geom. Funct. Anal.}\textbf{ 13} (2003), no. 1, 73--146.
\bibitem{Guentner} {\sc E. Guentner, N. Higson and S. Weinberger}, The Novikov conjecture for linear groups.
{\sl Publ. Math. Inst. Hautes Études Sci.} (2005), no. 101, 243--268.
\bibitem{Haagerup} {\sc U. Haagerup},   All nuclear $C^\star$-algebras are amenable. {\sl Inventiones mathematicae}, \textbf{74}, (1983), 305-319.
\bibitem{Higson} {\sc N. Higson and J. Roe}, Amenable group actions and the Novikov conjecture. {\sl J. Reine
Angew. Math. }\textbf{ 519} (2000), 143--153.
\bibitem{Joseph} {\sc M. Joseph}, Amenable wreath products with non almost finite actions of mean dimension zero. {\sl Trans. of the Amer. Math. Soc.} \textbf{377} (2024), 1321-1333.
\bibitem{Juschenko} {\sc K. Juschenko},  Amenability of discrete groups by examples.
{\sl Math. Surveys Monogr.,} \textbf{266}
American Mathematical Society, Providence, RI, (2022).
\bibitem{Kechris} {\sc A. S. Kechris and B. D. Miller,} Topics in orbit equivalence.
Lecture Notes in Mathematics 1852 (2004), Springer-Verlag, Berlin.
\bibitem{Kerr} {\sc D. Kerr},  Dimension, comparison, and almost finiteness.
{\sl J. Eur. Math. Soc.} \textbf{22} (2020), no. 11, 3697--3745.
\bibitem{Kesten} {\sc H. Kesten}, Full Banach mean values on countable groups. {\sl Math. Scand.} \textbf{7} (1959), 146--156.
\bibitem{Kirchberg} {\sc E. Kirchberg and N. C. Phillips}, Embedding of exact C*-algebras in the Cuntz algebra $\cO_2$. 
{\sl J. Reine Angew. Math.} \textbf{525} (2000), 17--53.
\bibitem{Lovasz} {\sc L. Lov\'asz}, Hyperfinite graphings and combinatorial optimization. 
{\sl Acta Math. Hung.}\textbf{ 161 }(2020), no. 2, 516--539.
\bibitem{Ma} {\sc X. Ma}, Fiberwise amenability of ample \'{e}tale groupoids. \emph{arXiv preprint arXiv:2110.11548.}
\bibitem{Mawu} {\sc X. Ma and J. Wu}, Almost elementariness and fiberwise amenability for \'etale groupoids. \emph{arXiv preprint arXiv:2011.01182.}
\bibitem{Matui}  {\sc H. Matui}, Homology and topological full groups of étale groupoids on totally disconnected spaces.
{\sl Proc. Lond. Math. Soc.} (3) {\bf 104} (2012), no. 1, 27--56.
\bibitem{Moharweiss} {\sc B. Mohar and W. Woess}, A survey on spectra of infinite graphs. {\sl Bull. London Math. Soc.}\textbf{ 21}
(1989) 209--234.
\bibitem{Murray} {\sc F.J. Murray and J. von Neumann}, On rings of operators IV {\sl Ann. of Math.} (2),\textbf{ 44 }(1943), 716–808. 
\bibitem{vonNeumann} {\sc J. von Neumann}, Zur allgemeinen Theorie des Maßes. {\sl Fund. Math.}, \textbf{13} (1), 
73--111.
\bibitem{Nowak} {\sc P. Nowak and G. Yu, } What is … property A? {\em Notices Amer. Math. Soc.} \textbf{55} (2008), no. 4, 474–475.
\bibitem{Osajda} {\sc D. Osajda}, 
Residually finite non-exact groups. {\sl Geom. Funct. Anal.}\textbf{ 28}(2018), no.2, 509--517.
\bibitem{Ozawa} {\sc N. Ozawa}, 
Amenable actions and exactness for discrete groups.
{\sl C. R. Acad. Sci. Paris Sér. I Math.}\textbf{ 330 }(2000), no. 8, 691--695.
\bibitem{OWE} {\sc D.S.  Ornstein and B. Weiss}, Entropy and isomorphism theorems for actions of amenable groups. {\sl Journal d'Analyse Mathématique}, 48(1), (1987), 1-141.
\bibitem{Roe} {\sc J. Roe}, Hyperbolic groups have finite asymptotic dimension. {\sl Proc. Amer. Math. Soc.}\textbf{ 133} (2005), no.9, 2489–2490.
\bibitem{Zivny} {\sc M. Romero, M. Wrochna and S. \v{Z}ivn\'y},
Treewidth-Pliability and PTAS for Max-CSP's.
{\em Proceedings of the 2021 ACM-SIAM Symposium on Discrete Algorithms (SODA) [Society for Industrial and Applied Mathematics (SIAM)] }(2021), 473-483.
\bibitem{Sako} {\sc H. Sako}, Property A and the operator norm localization property for discrete metric spaces.
{\sl J. Reine Angew. Math.}\textbf{ 690} (2014), 207--216.
\bibitem{Sims} {\sc A. Sims},  Hausdorff ´etale groupoids and their $C^*$-algebras, to appear in Operator algebras and dynamics: Groupoids, Crossed Products and Rokhlin dimension (A. Sims, G. Szab\'o, D. Williams and F.Perera (Ed.)) in Advanced Courses
in Mathematics. CRM Barcelona, Birkhauser. (2020)
\bibitem{Suzuki} {\sc Y. Suzuki}, Almost finiteness for general étale groupoids and its applications to stable rank of crossed products.
{\sl International Mathematics Research Notices}, (2020),\textbf{19} , 6007-6041.
\bibitem{quasi} {\sc A. Tikuisis, S. White and W. Winter}, Quasidiagonality of nuclear  $C^*$-algebras.
{\sl Ann. of Math. (2)} \textbf{185 }(2017), no. 1, 229--284.
\bibitem{Tu1} {\sc J-L. Tu}, Remarks on Yu's "property A'' for discrete metric spaces and groups. {\em Bull. Soc. Math. France} \textbf{129} (2001), no. 1, 115–139.
\bibitem{Tu2} {\sc J-L. Tu}, La conjecture de Baum–Connes pour les feuilletages moyennables.{\em K-Theory} \textbf{17}, (1999), 215–264. 
\bibitem{Weiss} {\sc B. Weiss}, Monotileable amenable groups. {\sl Translations of the American Mathematical Society-Series 2}\textbf{ 202,} (2001) 257-262.
\bibitem{Yu} {\sc G. Yu}, The coarse Baum-Connes conjecture for spaces which admit a uniform embedding into Hilbert space. {\em Invent. Math. } \textbf{139} (2000), no. 1, 201-240. 
\end{thebibliography}
\end{document}